\documentclass[12pt]{amsart}
\usepackage[a4paper,width=160mm,height=257mm,centering]{geometry}
\usepackage{amssymb,amsmath,amsthm,amscd,mathrsfs}
\usepackage{pstricks,pst-plot}
\usepackage{pb-diagram}

\parskip=\medskipamount

\theoremstyle{plain} %% This is the default
\newtheorem{theorem}{Theorem}[section]
\newtheorem{corollary}[theorem]{Corollary}
\newtheorem{lemma}[theorem]{Lemma}
\newtheorem{proposition}[theorem]{Proposition}
\newtheorem{problem}[theorem]{Problem}
\newtheorem{remark}[theorem]{Remark}

\newtheorem{example}{Example}

\theoremstyle{definition}
\newtheorem*{defn}{Definition}

\parskip=\smallskipamount
\numberwithin{equation}{section}

\newcommand{\B}{\mathbb{B}}
\newcommand{\C}{\mathbb{C}}
\newcommand{\D}{\mathbb{D}}

\newcommand{\N}{\mathbb{N}}
\renewcommand{\P}{\mathbb{P}}
\newcommand{\Q}{\mathbb{Q}}
\newcommand{\R}{\mathbb{R}}
\renewcommand{\S}{\mathbb{S}}

\newcommand{\Z}{\mathbb{Z}}

\newcommand{\CP}{\mathbb{C}\mathbb{P}}

\newcommand{\I}{\mathrm{i}}

\newcommand{\x}{\mathbf{x}}
\newcommand{\s}{\mathbf{s}}

\def\im{\mathop{\rm im}\nolimits}

\newcommand{\cC}{\mathscr{C}}
\newcommand{\cH}{\mathscr{H}}

\newcommand{\cJ}{\mathscr{J}}
\newcommand{\cO}{\mathscr{O}}

\newcommand{\setof}[1]{\left\{#1\right\}}
\newcommand{\wt}{\widetilde}
\newcommand{\dist}{\mathrm{dist}}

\def\ccup{\smallsmile}
\def\ccap{\smallfrown}

\def\Ext{\mathrm{Ext}}
\def\Hom{\mathrm{Hom}}
\def\colim{\mathop{\rm colim}\nolimits}

\begin{document}
\title[On the Hodge conjecture for $q$-complete manifolds]
{On the Hodge conjecture for $\mathbf q$-complete manifolds}

\author{Franc Forstneri\v c}
\author{Jaka Smrekar}
\author{Alexandre Sukhov}

\address{F.\ Forstneri\v c, Faculty of Mathematics and Physics, University of Ljubljana, and Institute of Mathematics, Physics and Mechanics, Jadranska 19, 1000 Ljubljana, Slovenia}
\email{franc.forstneric\,@fmf.uni-lj.si}

\address{J.\ Smrekar, Faculty of Mathematics and Physics, University of Ljubljana, and Institute of Mathematics, Physics and Mechanics, Jadranska 19, 1000 Ljubljana, Slovenia}
\email{jaka.smrekar@fmf.uni-lj.si}

\address{A.\ Sukhov, Universite Lille-1, Laboratoire Paul Painleve,
U.F.R.\ de Mathematiques, F-59655 Villeneuve d'Ascq Cedex, France}
\email{Alexandre.Sukhov@math.univ-lille1.fr}

\date{\today}

\begin{abstract}
We establish the Hodge conjecture for the top dimensional cohomology group with 
integer coefficients of  any $q$-complete complex manifold $X$ with $q<\dim X$.
This holds in particular for the complement $X=\CP^n\setminus A$ of any complex projective manifold
defined by $q<n$ independent equations.
\end{abstract}

\maketitle

%%%%%%%%%%%%%%%%%%%%%%%%%%%%%%%%%%%%%%%%%%%%%%%%%%%%%%%%%%%%%%%%%%%%%%%%%
%																			    %
%  Introduction                                                 										    %
%																                            %
%%%%%%%%%%%%%%%%%%%%%%%%%%%%%%%%%%%%%%%%%%%%%%%%%%%%%%%%%%%%%%%%%%%%%%%%%

\section{Introduction}
\label{sec:intro}

Every irreducible $p$-dimensional closed complex subvariety $Z$ in a compact complex manifold 
$X$ defines an integral homology class $[Z]\in H_{2p}(X;\Z)$  \cite{AH1961}. 
A finite linear combination $Z=\sum_j n_j Z_j$ of such subvarieties
with integer coefficients is called an {\em  analytic cycle} in $X$, and
the corresponding homology class $z=\sum_j n_j z_j\in H_{2p}(X;\Z)$ is an
{\em analytic homology class}. A cohomology class $u \in H^{2k}(X;\Z)$ is said to be (complex) 
{\em analytic} if $u$ corresponds under Poincar\'e duality 
to an analytic homology class $z\in H_{2p}(X;\Z)$ with $p=\dim X -k$. 
The same notions can be considered with rational coefficients $n_j\in\Q$. 
If $X$ is compact K\"ahler, we have % the Hodge decomposition 
$H^{2k}(X;\C)=\oplus_{i+j=2k} H^{i,j}(X)$, and the image in $H^{2k}(X;\C)$ 
of any analytic cohomology class belongs to $H^{k,k}(X)$. 
The {\em Hodge conjecture} \cite{Hodge1952} states that every rational class
$u\in H^{2k}(X;\Q)\cap H^{k,k}(X)$ of type $(k,k)$ of a compact projective manifold 
is analytic, hence algebraic by Chow's theorem.
(Hodge's original conjecture for integer coefficients fails both in the projective  
and in the Stein case, see Atiyah and Hirzebruch \cite{AH1961} and 
Grothendieck \cite{Grothendieck1969}.)
Although the Hodge conjecture spawned a great number of works 
(cf.\ \cite{AH1961,Buh1970,CG1975,Deligne,Demailly2012,SouleVoisin}
and the references therein), it remains open.

Assume now that $M$ is a compact complex manifold of dimension $n$ 
with boundary $\partial M$. In view of the Poincar\'e-Lefschetz duality 
\begin{equation}
\label{eq:PD1}
	H^{k}(M;G) \cong H_{2n-k}(M,\partial M;G),  \quad k=0,1,\ldots,2n
\end{equation} 
(see e.g.\ \cite[Theorem 7.7, p.\ 227]{Massey} or \cite[Theorem 20, p.\ 298]{Spanier}), 
the Hodge conjecture amounts to asking whether every cohomology class in 
$H^{2k}(M;\Z)$ or $H^{2k}(M;\Q)$ can be represented by a {\em relative analytic cycle}
consisting of closed complex subvarieties of $M$ with boundaries in $\partial M$. 

Our first main result is the following. We use the notion of $q$-completeness due to 
Grauert \cite{Grauert1959,Grauert:q-convexity}, see below for more information.

%
%
%   HODGE1
%
%
\begin{theorem}\label{th:Hodge1}
Let $X$ be a complex manifold of dimension $n>1$ and $M\subset X$ be
a compact $q$-complete domain for some $q\in\{1,\ldots,n-1\}$.
If  the number $n+q-1=2k$ is even, then every cohomology class in $H^{2k}(M;\Z)$
is Poincar\'e dual to an  analytic cycle $Z=\sum_j n_j Z_j$ of dimension 
$p=n-k$, where each $Z_j$ is an embedded complex submanifold of $M$ 
(immersed with normal crossings if $q=1$) with smooth boundary $\partial Z_j\subset \partial M$.
\end{theorem}

If we allow the $Z_j$'s to have boundaries in a collar around $\partial M$, 
the cycle representing a cohomology class in $H^{2k}(M;\Z)$ can be chosen to
consist of holomorphic images of the unit ball in $\C^p$ with $p=n-k$ 
(see Theorem \ref{th:Hodge2} and Corollary \ref{cor:Hodge2}). 
We also prove the corresponding result for $q$-complete manifolds without boundary,
possibly with infinite topology; see Theorem \ref{th:Hodge3}. 

Before proceeding, we recall the notion of $q$-convexity and $q$-completeness. 
(See Grauert \cite{Grauert1959,Grauert:q-convexity}. A different version of these properties
was introduced by Henkin and Leiterer \cite[Def.\ 4.3]{HL-AG};
here we use Grauert's original definitions.) 

Let $X$ be a complex manifold of dimension $n>1$ and $q\in \N$ be an integer. 
A smooth real function $\rho\colon X\to \R$  is said to be 
{\em $q$-convex} on an open set $\Omega \subset X$ if  the Levi form of $\rho$
has at least $n-q+1$ positive eigenvalues  at every point of $\Omega$. 
(Recall that the Levi form is the hermitian  quadratic form on the holomorphic tangent bundle 
$TX$ which is given in any local holomorphic coordinates
$z=(z_1,\ldots,z_n)$ on $X$ by $L_\rho(z)w = \sum_{j,k=1}^n \frac{\partial^2 \rho(z)}{\partial z_j\partial \bar z_k} w_j\bar w_k$.) The manifold $X$ is $q$-convex if it admits a smooth exhaustion function  
$\rho \colon X \to\R$  which is  $q$-convex on $X\setminus K$ for some compact set $K\subset X$;
$X$ is {\em $q$-complete} if $\rho$ can be chosen $q$-convex on all of $X$. 
A compact domain $M\subset X$ with smooth boundary $\partial M$ is 
said to be $q$-complete if $M=\{\rho\le 0\}$, where $\rho$ is a $q$-convex function 
on a neighborhood of $M$ without critical points on $\partial M=\{\rho=0\}$. 
In particular, if $\rho \colon X \to\R$ is a $q$-convex exhaustion function then 
for any regular value $c$ of $\rho$ the sublevel set $\{\rho\le c\}$ is a  $q$-complete domain in $X$. 
Note that a 1-convex function is a strongly plurisubharmonic function, a 1-complete manifold
is a {\em Stein manifold} (Grauert \cite{Grauert1958}), and a $1$-complete domain 
is a Stein domain with strongly pseudoconvex boundary. 
Every complex manifold of dimension $n$ is trivially $(n+1)$-complete, and is $n$-complete 
if it has no compact connected components (see Ohsawa \cite{Ohsawa1984} and Demailly \cite{Demailly1990}). 
In particular, the manifolds considered in this paper are never compact.
It follows from the maximum principle for subharmonic functions that a $q$-complete manifold does
not contain any compact complex submanifolds of dimension $\ge q$ and this bound is sharp
in general (see Example \ref{ex:projective} in  Sec.\  \ref{sec:examples}).

The most  interesting examples of $q$-convex manifolds for $q>1$ arise as complements
of complex subvarieties. For  example, the complement $\CP^n\setminus A$ of any compact projective
manifold $A \subset \CP^n$ of complex codimension $q$ is $q$-convex
(Barth \cite{Barth1970}). The same holds for the complement of any compact complex submanifold
with Griffiths positive normal bundle in an arbitrary compact complex manifold (Schneider \cite{Schneider}). 
In general $\CP^n\setminus A$ is not $q$-complete but is $(2q-1)$-complete (Peternell \cite{Peternell}); 
if however $A$ is defined by $q$ global equations in $\CP^n$ then $\CP^n\setminus A$
is also $q$-complete. For example, $\CP^n\setminus\CP^{n-q}$ is $q$-complete 
for any pair of integers $1\le q\le n$. Further, if $Y$ is a compact complex manifold, $L\to Y$ 
is a positive holomorphic line bundle, and $s_1,\ldots, s_q:Y \to L$ are holomorphic sections
whose common zero set $A=\{s_1=0,\ldots, s_q=0\}$   
has codimension $q$ in $Y$, then $Y\setminus A$ is $q$-complete 
(see Andreotti and Norguet \cite{AN1966}). A more complete list of examples 
can be found in \cite{Grauert:q-convexity} and in the survey by Col\c toiu \cite{Coltoiu1995}.

The group $H^{n+q-1}(M;\Z)$ appearing in Theorem \ref{th:Hodge1} is the top dimensional  
a priori nontrivial cohomology group of a $q$-complete manifold. Indeed, 
$q$-convexity of a function is a stable property in the fine $\cC^2$ Whitney topology, 
so every $q$-complete manifold admits a $q$-convex {\em Morse} exhaustion function.  
The Morse index of any critical point of such a function is $\le n+q-1$ 
(see \cite[p.\ 91]{FF:book} for the quadratic normal form), 
so it follows from Morse theory that a $q$-complete manifold $M$ of dimension $n$ 
is a handlebody with handles of indices at most $n+q-1$. In particular, for any abelian group $G$
we have
\begin{equation}\label{eq:coh-vanishes}
		H^k(M;G) = 0\qquad \forall k>n+q-1.
\end{equation}

Before discussing further results, we compare 
Theorem \ref{th:Hodge1} with the known results in the literature,
indicate why the standard proofs do not apply in our situation, and 
outline the method that we introduce to address this problem.

There are relatively few results concerning the Hodge conjecture for noncompact manifolds.
For a Stein manifold $X$  (the case $q=1$ of Theorem \ref{th:Hodge1}), the Hodge conjecture 
holds for all cohomology groups $H^{2k}(X;\Q)$ with rational coefficients, 
but it fails in general for integer coefficients; see Atiyah and Hirzebruch \cite{AH1961}, 
Buh\v staber \cite{Buh1970},  and Cornalba and Griffiths \cite{CG1975}.
(It follows from Oka's theorem on complex line bundles that the Hodge conjecture 
holds for the lowest dimensional integral cohomology group $H^2(X;\Z)$, 
cf.\ Kodaira and Spencer \cite{KS1953}.)
Atiyah and Hirzebruch \cite{AH1960}  showed  that on any complex manifold $X$,
a necessary condition for a cohomology class $z \in H^{2k}(X;\Z)$ to be analytic 
is that $z$ lies in the kernels of all differentials of the Atiyah-Hirzebruch spectral sequence associated to $X$. 
Any given cohomology class $z$ satisfies this necessary condition after multiplication by some integer $N(z)\in \N$. 
An explicit expression for $N(z)$ was computed by Buh\v staber \cite{Buh1970}; it follows from 
his result that if $z$ belongs to the top dimensional nontrivial cohomology group of a $q$-complete manifold, 
then $N(z) = 1$. This explains why in the present article we do not need to consider cohomologies with 
rational coefficients. In fact, the special case $q=1$ of our Theorem \ref{th:Hodge1} corresponds 
to Buh\v staber's result in this top dimension. On a Stein manifold, the necessary condition 
of Atiyah and Hirzebruch  is also sufficient (see Cornalba and Griffiths \cite{CG1975} for a detailed proof). 
However, as mentioned by Deligne \cite{Deligne},  it is not known whether the Atiyah-Hirzebruch  
necessary condition is also sufficient on non-Stein manifolds. 
Our results give the  affirmative answer  to this question for the top dimensional cohomology 
group of a $q$-complete manifold.

Differentials of the Atiyah-Hirzebruch spectral sequence annihilate any non-torsion
element of a cohomology group. From this point of view it is natural to give an example 
of a $q$-complete manifold $X$ with a nontrivial torsion part in the top cohomology 
$H^{n+q-1}(X;\Z)$. Here is an example with $n=2$ and $q=1$,
i.e., a Stein surface. This is a special case of the example provided by
Proposition \ref{prop:coh} in Sect.\ \ref{sec:examples}.

\begin{example} \label{ex:Ex1}
Consider a smooth complex curve $C$ in $\CP^2$ of degree $d > 1$ and genus 
$g=(d-1)(d-2)/2$. Its complement
$X=\CP^2\setminus C$ is a Stein surface and we have that
\[
	H^{2}(X;\Z)=\Z_d \oplus  \Z^{\beta_1}= \Z_d \oplus  \Z^{2g}
\] 
where  $\beta_1=2-\chi(A)=2g$ is the first Betti number of $A$. 
\end{example}

By contrast, the top absolute homology group $H_{n+q-1}(X;\Z)$ is always free.
In the Stein case this is a classical result of Andreotti and Frankel \cite{AF1959};
see also Andreotti and Narasimhan \cite{AN1962} and Hamm \cite{Hamm1} for Stein spaces.
The result was generalized to the $q$-complete case by Sorani \cite{So1962} and Hamm \cite{Hamm2}.

The proofs  in \cite{AH1961,Buh1970,CG1975}
proceed by representing even dimensional cohomology classes by  Chern classes of complex vector bundles.
The Oka-Grauert principle \cite{GrauertOP} implies that every complex vector bundle over a Stein manifold 
admits a compatible  holomorphic vector bundle structure. (See also \cite{FF:book}.)
The zero set of a generically chosen holomorphic section of such a bundle
is an analytic cycle that is Poincar\'e dual to its Chern class. 
A similar approach is used on compact K\"ahler manifolds with ample holomorphic line bundles. 
Analytic cycles obtained in this way are given by holomorphic equations, so
we have no  information on the complex structure of their irreducible components.

The Oka-Grauert principle fails in general on $q$-convex manifolds for $q>1$.
In the present paper we introduce a completely different method which relies on the technique of constructing proper holomorphic maps, immersions and embeddings of strongly pseudoconvex Stein domains
to $q$-complete manifolds, due to Drinovec Drnov\v sek and Forstneri\v c \cite{DF2007,DF2010}
(see Theorem \ref{th:lifting} in Sec.\ \ref{sec:lifting}). We work with holomorphic maps from 
specific strongly pseudoconvex domains into $M$, inductively stretching their boundaries 
towards the boundary $\partial M$. This technique does not rely on the function theory of the target manifold, 
but only of the source strongly pseudoconvex domain where solutions of the 
$\overline\partial$-problem are readily available. This allows
us to keep the same complex structure on these domains during the entire process.
At every critical point of maximal index $n+q-1$ of the exhaustion function on $M$,
a new component of the analytic cycle may appear, given as a cross-section of the corresponding 
handle (see Proposition \ref{cor:change}). This cross-section can be realized in 
local holomorphic coordinates as the intersection of a $\C$-linear subspace $L\subset\C^n$
of complex dimension $p=(n-q+1)/2$ with a thin round tube around the core $E$ of the handle.
Such intersection is an ellipsoid in $\C^p$ (the image of the unit ball in $\C^p$ 
under an $\R$-linear automorphism), so we can obtain a representation of
cohomology classes in $H^{n+q-1}(M;\Z)$ by analytic cycles 
whose irreducible components are normalized by ellipsoids.  However, 
it turns out that the precise choice of a domain in $L$ will not be important as long as it is
small enough, strongly pseudoconvex, and it contains the intersection point $E\cap L$.
In particular, we are free to choose a ball in $L\cong\C^p$; this leads to 
analytic cycles whose irreducible components are embedded or immersed copies 
of the ball $\B^p$. Due to rigidity phenomena in Cauchy-Riemann geometry, one can not 
push the boundaries of these balls exactly into the boundary of $M$. We choose instead  an interior collar 
$A\subset M$ around the boundary $\partial M$, that is, a compact neighborhood of $\partial M$ in $M$ 
homeomorphic to $\partial M\times [0,1]$, with $\partial M = \partial M\times \{0\} \subset A$. 
The complement $N=\overline{M\setminus A}$ is then a compact manifold with boundary 
in $\mathring M$ that is homeomorphic to $M$. Since the inclusions $N\hookrightarrow M$ 
and $\partial M\hookrightarrow A$ are homotopy equivalences, they induce isomorphisms 
\begin{equation}\label{eq:collar}
	 H^k(M;G)\cong H^k(N;G),\qquad H_k(M,\partial M; G)\cong H_k(M,A;G).
\end{equation}
(See Lemma \ref{lem:collar}.)  By the Lefschetz-Poincar\'e duality we have
\begin{equation}\label{eq:PD}
		H^{k}(N;G) \cong H_{2n-k}(M,M\setminus N;G)
\end{equation}
(see \cite[Proposition 6.4, p.\ 221]{Massey}). From (\ref{eq:collar}) and (\ref{eq:PD}) it follows that 
\[
		H^k(M;G)\cong H_{2n-k}(M,A;G); 
\]
that is, cohomology classes of $M$ can be represented by cycles with boundaries in a collar 
around $\partial M$. In light of this, we have the following version of Theorem \ref{th:Hodge1}.

%
%
%   HODGE2
%
%
\begin{theorem}\label{th:Hodge2}
Let $M$ be a compact $q$-complete domain in a complex manifold $X$ of dimension $n$,
where $q\in\{1,\ldots,n-1\}$, and let $A\subset M$ be a collar around $\partial M$. 
If  the number $n+q-1=2k$ is even, then every class in $H^{2k}(M;\Z)$
is represented by an analytic cycle $Z=\sum_j n_j Z_j$ of dimension 
$p=n-k=(n-q+1)/2$ with integer coefficients, where each $Z_j$ is an embedded complex submanifold 
of $M$ with smooth boundary $\partial Z_j\subset A$ (immersed with normal crossings if $q=1$) 
that is biholomorphic to the ball $\B^p\subset \C^p$.
\end{theorem}

When $n=2$ and $q=1$ (that is, when $M$ is a strongly pseudoconvex domain
in a Stein surface), we get cycles consisting of analytic discs, i.e., 
holomorphic images of the unit disc $\D=\{\zeta \in \C\colon |\zeta| <1\}$. Since one can always
push the boundaries of discs precisely into the boundary of a strongly 
pseudoconvex domain  (see \cite{FG1992}), we obtain the following corollary
to Theorem \ref{th:Hodge2}. (See also J\"oricke \cite{Joricke2009}, especially Corollary 3 on p.\ 78.)

%
%
%   COLROLLARY HODGE2
%
%
\begin{corollary}\label{cor:Hodge2}
If $M$ is a strongly pseudoconvex Stein domain of dimension $2$, then every class in
$H^2(M;\Z)$ can be represented by an analytic cycle whose irreducible components
are properly immersed discs with normal crossings that are smooth up to the boundary.
\end{corollary}

We now explain a version of Theorem \ref{th:Hodge1} for $q$-complete manifolds
$X$ without boundary, possibly with infinite topology. 

Let $\rho\colon X\to\R$ be a $q$-convex Morse exhaustion function. Choose an exhaustion 
$M_1\subset M_2\subset\cdots \subset\bigcup_{j=1}^\infty M_j = X$, where each 
$M_j$ is a regular sublevel set $M_j =\{x\in X\colon \rho(x)\le c_j\}$ and 
there is at most one critical point of $\rho$ in each difference $\mathring M_j\setminus M_{j-1}$.
The inclusion $M_j \hookrightarrow M_{j+1}$ induces a homomorphism 
$H^{k}(M_{j+1};G)\to H^{k}(M_j;G)$, and we have a well defined inverse limit
\begin{equation}\label{eq:invlim}
	\cH^{k}(X;G) = \lim_j H^{k}(M_j;G).
\end{equation}
This definition is due to Atiyah and Hirzebruch \cite{AH1961}; see Section \ref{sec:openMfd} 
below for the details.
There is a natural surjective homomorphism $H^k(X;G)\to \cH^{k}(X;G)$ from the singular cohomology whose
kernel can be described by means of the first derived functor of the inverse limit. This
kernel is trivial when $G$ is a field (e.g., when $G=\Q$) and also when $G=\Z$ and the homology group
$H_{k-1}(X)$ is not too bad. We discuss this in Section \ref{sec:equality}.

An irreducible closed subvariety $Z$ of dimension $p$ in $X$ with small singular locus
defines a cohomology class in $H^{2n-2p}(X;\Z)$. (If $Z$ is smooth, this is 
the Thom class of $Z$ in $X$ made absolute; see Section \ref{sec:openMfd} below for the general case.)
In this setting, we say that the corresponding cohomology class in $\cH^{2n-2p}(X;\Z)$ is an analytic
class represented by $Z$.

%
%
%   HODGE3
%
%
\begin{theorem}  \label{th:Hodge3}
Let $X$ be a complex manifold of dimension $n>1$ which is $q$-complete for some $q\in \{1,\ldots,n-1\}$.
If the number $n+q-1=2k\ge 2$ is even, then every class in $\cH^{2k}(X;\Z)$ is 
represented by an analytic cycle $\sum_r n_r Z_r$  with integer coefficients, where
each $Z_r$ is a properly embedded (immersed with normal crossings if $q=1$) complex submanifold
of $X$ biholomorphic to the ball of dimension $p=n-k=(n-q+1)/2$.
\end{theorem}

The cycle $\sum_r n_r Z_r$ in Theorem \ref{th:Hodge3} is infinite (but locally finite) in general;
it can be chosen finite if $X$ admits  a $q$-convex exhaustion function with finitely many critical points.

\begin{remark}\label{rem:finitely-gen}
If the homology group $H_{2k-1}(X;\Z)$ is the direct sum of a free abelian group and a torsion
abelian group, for example, if it is finitely generated, then the natural morphism 
$H^{2k}(X;\Z)\to\cH^{2k}(X;\Z)$ is an isomorphism (cf.\ Proposition \ref{prop:vanishing}).
Hence in this case Theorem \ref{th:Hodge3} applies to every class in $H^{2k}(X;\Z)$
(see Corollary \ref{cor:vanishing}).
\end{remark}

%
%
%  REMARK
%
%
\begin{remark}\label{rem:handles}
Our proof will show that, in the absence of critical points of index $>m$ for some even integer 
$0\le m \le n+q-1$, Theorems \ref{th:Hodge1}, \ref{th:Hodge2}, and
\ref{th:Hodge3}  hold for the top dimensional nontrivial 
cohomology group $H^{m}$. In particular, if $X$ is an odd dimensional Stein manifold
which is subcritical, in the sense that it admits a strongly plurisubharmonic exhaustion function
$\rho\colon X\to\R$ without critical points of index $n=\dim X$, then Theorem
\ref{th:Hodge3} holds for the group $H^{n-1}(X;\Z)$.
\end{remark}

Our method also applies to analytic cycles of lower dimension, 
but these are Poincar\'e dual to higher cohomology classes which are trivial in view of
 (\ref{eq:coh-vanishes}). On the other hand, it does not work  for cycles of dimension $>n-q+1$, 
and hence for cohomology groups of dimension $<n+q-1$, 
because we are unable to push the boundaries of such cycles across 
the critical points of index $n+q-1$ of a $q$-convex exhaustion function.
This difficulty is not only apparent as shown by examples in \cite{DF2010},
and is further demonstrated by the fact that even in the Stein case $(q=1)$, 
the analogues of Theorems \ref{th:Hodge1} and \ref{th:Hodge3} 
for integer coefficients fail in general \cite{Buh1970}. 
It is a challenging problem to give a proof of the corresponding results from \cite{Buh1970,CG1975} 
for lower dimensional cohomology groups with rational coefficients avoiding
the use of Chern classes and the Oka-Grauert principle. By finding such proof, one
might hope to answer the following question.

\begin{problem}
Assume that $X$ is an $n$-dimensional complex manifold which is $q$-complete for some
$q\in \{1,\ldots,n-1\}$. Does the Hodge conjecture hold for the cohomology groups
$H^{2k}(X;\Q)$ when $2\le 2k< n+q-1$?
\end{problem}

The Hodge conjecture has also been considered in the category of symplectic manifolds.
Donaldson \cite{Donaldson1996,Donaldson1996-2} constructed symplectic submanifolds
of any even codimension in a given compact symplectic manifold $(X,\omega)$ of dimension $2n\ge 4$.
In particular, he showed that if the cohomology class $[\omega/2\pi]\in H^2(X;\R)$ admits a lift
to an integral class  $h\in H^2(X;\Z)$, then for any sufficiently large integer $k\in\N$ the Poincar\'e
dual of $kh$ in $H_{2n-2}(X;\Z)$ can be represented by a compact symplectic submanifold 
of real codimension two in $X$.  In light of this it is natural to ask the following question.

\begin{problem}
It is possible to represent certain cohomology classes of  noncompact symplectic manifolds
by noncompact cycles consisting of proper symplectic submanifolds? 
\end{problem}

%%%%%%%%%%%%%%%%%%%%%%%%%%%%%%%%%%%%%%%%%%%%%%%%%%%%%%%%%%%%%%%%%%%%%%%%%
%																			    %
%  Topological preliminaries                                  										    %
%																                            %
%%%%%%%%%%%%%%%%%%%%%%%%%%%%%%%%%%%%%%%%%%%%%%%%%%%%%%%%%%%%%%%%%%%%%%%%%

\section{Topological preliminaries}
\label{sec:top-prelim}
In this section we review the necessary topological background.
In particular, Proposition \ref{cor:change} below gives a precise description of the 
effect of a handle attachment on the relative homology (and, by the Poincar\'e-Lefschetz duality,
on the cohomology) of a compact manifold with boundary. 
Although this is standard, we need very precise geometric information 
on the cycles generating the relative homology group. For this reason, and lacking 
a precise reference, we provide a detailed proof.

Denote by $\B^k$ the closed ball in $\R^k$ and by $\S^{k-1}=\partial \B^k$ the $(k-1)$-sphere. Let $M$ be an orientable closed $n$-manifold with boundary $\partial M$, and let 
\begin{equation} \label{eq:phi}
	\phi\colon\S^{k-1}\times\B^{n-k}\hookrightarrow\partial M
\end{equation}
be the attaching map of a $k$-handle $H=\B^k\times \B^{n-k}$. We assume that $\phi$ is a homeomorphism
onto its image. Set $N=M\cup_\phi H$; this is a compact manifold with boundary
\[
		\partial N=(\partial M\setminus \im \phi) \cup (\B^{k}\times\S^{n-k-1}).
\]

\begin{proposition}\label{cor:change}
Assume that $(M,\partial M)$ is a compact manifold with boundary, and let $N$ be obtained by adding
a handle of index $k$ to $M$ by an attaching map $\phi$ (\ref{eq:phi}). Assume that for some $j\in\{1,\ldots,n-k\}$
the group $H_j(M,\partial M)$ can be realized by a collection $\cC$ of geometric $j$-cycles in $(M,\overline{\partial M\setminus\im\phi})$.
Then $H_j(N,\partial N)$ can be realized by the cycles in $\cC$ prolonged by inclusion, and, if $j=n-k$, possibly by an additional relative
disc in any fibre of $H$ (viewed as a $j$-disc bundle over $\B^k$).
\end{proposition}

Before proving the proposition we recall some preliminary material.

Given pairs of topological spaces $A \subset M$ and $B\subset N$, the notation
$f\colon (M,A)\to (N,B)$ means that $f\colon M\to N$ is a continuous map satisfying $f(A)\subset B$. 
A a similar notation is used for maps of triads; thus $f\colon (M;A_1,A_2)\to (N;B_1,B_2)$
means that $f:M\to N$ is a continuous map satisfying $f(A_i)\subset B_i$ for $i=1,2$.

By $\ccap$ and $\ccup$ we denote the cap and the cup
product on cohomology, respectively. For general background to homology 
and cohomology we refer to Bredon \cite{Bredon}  or Spanier \cite{Spanier}.

We recall the following naturality property of the cap-product (see Spanier, 5.6.16).
This holds for an arbitrary coefficient group which we omit from the notation.

\begin{lemma}\label{cap_naturality}
Let $f\colon M\to N$ map a subset $A_1\subset M$ to $B_1$ and $A_2\subset M$ to $B_2$, that is, $f\colon(M;A_1,A_2)\to(N;B_1,B_2)$ is a map of triads. Let $u\in H^q(N,B_1)$ and $z\in H_n(M,A_1\cup A_2)$. 
Let $f_1\colon(M,A_1)\to(N,B_1)$, $f_2\colon(M,A_2)\to(N,B_2)$ and 
$\Bar f\colon(M,A_1\cup A_2)\to(N,B_1\cup B_2)$ be maps defined by $f$. 
The following relation holds in $H_{n-q}(N,B_2)$:
\[ 
	f_2{}_{*}\left(f_1^*u\ccap z\right)=u\ccap\Bar f_*z. 
\]
\end{lemma}

In other words, Lemma \ref{cap_naturality} renders the following diagram commutative.
\begin{equation*}\begin{diagram}
 \node{H^q(N,B_1)}	\arrow{e,t}{f_1^*}	\arrow{s,l}{u\mapsto u\ccap\Bar f_*(z)}	
 \node{H^q(M,A_1)} \arrow{s,r}{v\mapsto v\ccap z} \\
 \node{H_{n-q}(N,B_2)}	\node{H_{n-q}(M,A_2)} \arrow{w,t}{f_2{}_*}
\end{diagram}\end{equation*}

Let $f\colon(M;\im\phi,\overline{\partial M\setminus\im\phi})\to(N;H,\partial N)$ and 
$J\colon(N;\emptyset,\partial N)\to(N;H,\partial N)$ be inclusions of triads, and let 
$\Bar f\colon(M,\partial M)\to(N,H\cup\partial N)$ and $\Bar J\colon(N,\partial N)\to(N,H\cup\partial N)$
be the associated maps of pairs. We choose compatible orientation classes 
$\mu\in H_n(M,\partial M)$ and $\nu\in H_n(N,\partial N)$ in the
sense that $\Bar f_*(\mu)=\Bar J_*(\nu)\in H_n(N,H\cup\partial N)$. 
Consider the following diagram where the horizontal arrows are induced
by inclusions and the vertical arrows by cap-products.

\begin{equation*}
\divide\dgARROWLENGTH by2
\begin{diagram}
\node{H^q(N)} \arrow{s,r}{\ccap\nu} \node{H^q(N,H)} \arrow{w} \arrow{s,r}{\ccap \Bar J_*(\nu)}
	\arrow{e}	\node{H^q(M,\im\phi)}	\arrow{e,t}{i^*} \arrow{s,r}{\ccap\mu} \node{H^q(M)} \arrow{s,l}{\ccap\mu} \\
	\node{H_{n-q}(N,\partial N)} \arrow{e,=} \node{H_{n-q}(N,\partial N)} 
	\node{H_{n-q}(M,\overline{\partial M\setminus\im\phi})} \arrow{w}
	\arrow{e,t}{\iota_*}	\node{H_{n-q}(M,\partial M)}
\end{diagram}\end{equation*}

\begin{lemma}
\label{prop:com}
For any $q$, the above diagram is commutative. All the arrows, with the possible exception of $i^*$ and $\iota_*$, are isomorphisms when $q>0$.
\end{lemma}

\begin{proof}
Lemma \ref{cap_naturality} implies that the left-hand side square is commutative by virtue of the map of triads $J$,
and that the middle square is commutative by virtue of $f$. The right-hand side square is induced by the Poincar\'e-Lefschetz duality for the decomposition $\partial M=\overline{\partial M\setminus\im\phi}\cup\im\phi$ 
(\cite{Bredon}, Section VI.9, Problems); all cap-products with $\mu$ and $\nu$ are duality isomorphisms. 
The arrow $H^q(N,H)\to H^q(N)$ is an isomorphism when $q>0$ because $H$ is contractible. 
Using commutativity it follows first that also the cap-product with $\Bar f_*(\mu)=\Bar J_*(\nu)$
is an isomorphism, and second that the central bottom horizontal arrow is an isomorphism. Finally, the arrow $H^q(N,H)\to H^q(M,\im\phi)$ is an isomorphism because the quotients $N/H$ and $M/\im\phi$ are evidently homeomorphic.
\end{proof}

\begin{proof}[Proof of Proposition \ref{cor:change}]
The morphism $\iota_*$ in the above diagram sits in the exact sequence of the triad
$(M; \im\phi,\overline{\partial M\setminus \im\phi})$ as follows:
\begin{equation}\label{eq:triadSeq}	
	H_j(\im\phi,\partial(\im\phi))\to H_j(M,\overline{\partial M\setminus\im\phi}) 
	\xrightarrow{\iota_*} H_j(M,\partial M)\to H_{j-1}(\im\phi,\partial(\im\phi)).
\end{equation}
Here, the pair $(\im\phi,\partial(\im\phi))$ is seen to be the product of pairs $\S^{k-1}\times(\B^{n-k},\S^{n-k-1})$.

If $j<n-k$,  both endgroups in (\ref{eq:triadSeq}) vanish and therefore, $\iota_*$ is an isomorphism.

If $j=n-k$, the sequence (\ref{eq:triadSeq}) reads
\[ 
	\Z\to H_j(M,\overline{\partial M\setminus\im\phi})\xrightarrow{\iota_*} H_j(M,\partial M)\to 0.	
\]
Here, the group of integers $\Z$ is generated by any relative disc $D_\zeta=(\setof{\zeta}\times\B^j,\setof{\zeta}\times\S^{j-1})$ where $\zeta\in\S^{n-j-1}= \S^{k-1}$.
If $\iota_*$ is not an isomorphism, then, clearly, if we add to $\cC$ a relative disc $D_\zeta$ 
(which may generate a torsion element in $H_{j}(M,\overline{\partial M\setminus\im\phi})\cong H_j(N,\partial N)$), 
we obtain a collection of geometric generators for $H_j(M,\overline{\partial M\setminus\im\phi})$
which can be prolonged by inclusion to a set of geometric generators for $H_j(N,\partial N)$. 
We remark that in $H_j(N,\partial N)$, the relative discs $D_\zeta$ are homologous to 
the relative discs in the fibres of the handle $H$.
\end{proof}

\begin{remark}\label{rem:change}
The assumption in Proposition \ref{cor:change}, that the group $H_j(M,\partial M)$ 
can be realized by a collection of geometric $j$-cycles in 
$(M,\overline{\partial M\setminus\im\phi})$, is always satisfied when $j\le n-k$.
Indeed, the boundary sphere $\S^{k-1}=\partial \B^k$ of the core $k$-disc $\B^k$ 
of the handle has dimension $k-1$, while the boundaries of relative cycles representing 
the homology classes in $H_j(M,\partial M)$ have dimension $j-1$.
We can choose these cycles so that only their boundaries intersect $\partial M$.
Since $(j-1) +(k-1) \le (n-k-1)+(k-1) < n-1=\dim \partial M$, 
a general position argument shows that the (finitely many) generators of 
$H_j(M,\partial M)$ can be represented by relative cycles in 
$(M,\partial M)$ whose boundaries avoid the attaching sphere $\phi(\S^{k-1})\subset\partial M$. 
By choosing the transverse disc of the handle sufficiently thin
we can ensure that these cycles avoid the attaching set $\im \phi$ of the handle.  
\end{remark}

\begin{lemma}\label{lem:collar}
Let $M$ be a manifold and $A\subset M$ a collar of the boundary $\partial M$.
Then the inclusion induced homomorphisms $H_q(M,\partial M)\to H_q(M,A)$ are isomorphisms for all $q$.
\end{lemma}

\begin{proof}
Since the inclusion $\partial M\hookrightarrow A$ is a homotopy equivalence, this follows from
naturality of the long homology exact sequence of the pair in conjunction with the 5-lemma.
\end{proof}

Let $(X,A)$ be a pair of topological spaces. In the following proposition we use $C_p(X, A)$ for
the group of relative (singular) $p$-chains with integer coefficients. Let $K$ be an oriented compact
$p$-manifold in $X$ with $\partial K\subset A$. There is an orientation $p$-chain in $C_p(K,\partial K)$
which may be viewed also as a $p$-chain in $C_p(X,A)$. We denote the latter $p$-chain simply by $(K,\partial K)$.
We say that the corresponding element in $H_p(X,A)$ is {\it represented} by $(K,\partial K)$.

\begin{proposition}
\label{prop:extcollar}
Let $M$ be an oriented manifold with boundary $\partial M$. Assume that the manifold $M'$ is obtained from
$M$ by adding an exterior collar $A\cong\partial M\times[0,1]$, so that $\partial M\equiv\partial M\times\setof{0}$ and $\partial M'=\partial M\times\setof{1}$. Let $K$ be an oriented compact $p$-manifold in $M'$ whose boundary $\partial K$ is contained in $\partial M\times(0,1]$. Furthermore, assume that $K$ meets $\partial M$ transversely. Then the geometric intersection of $K$ with $M$ yields a finite collection $\setof{L_i}$ of connected oriented $p$-manifolds with $\partial L_i\subset\partial M$ so that the sum $\sum_i(L_i,\partial L_i)$ is a chain in $C_p(M,\partial M)$, homologous to $(K,\partial K)$ in $C_p(M',A)$. Consequently, if an element $z'$ of $H_p(M',A)$ is represented by a linear combination of $p$-manifolds satisfying the conditions on $K$, there is an element $z$ of $H_p(M,\partial M)$ which is also represented with a linear combination of $p$-manifolds and is mapped to $z'$ under the isomorphism $H_p(M,\partial M)\to H_p(M',A)$. 
\end{proposition}

\begin{proof}
As the compact manifold $K$ meets $\partial M$ transversely (and $\partial K\cap\partial M=\emptyset$), the intersection $K\cap\partial M$ is the union of a finite collection $\cC$ of closed $(p-1)$-manifolds. 
Note that the connected components of $K'=K\setminus\partial M$ are contained
either in $M$ or in $A$. Their closures are connected submanifolds whose boundary components are in $\cC$. Let $\setof{L_i}$ denote the collection of closures of components of $K'$ that are contained in $M$, and let $\setof{L_j'}$ 
denote the collection of closures of components of $K'$ that are contained in $A$. 
Assume that all are oriented compatibly with $K$. Tautologically, $\sum_i(L_i,\partial L_i)$
and $\sum_j(L_j',\partial L_j')$ can be viewed as chains in $C_p(M',A)$ whose sum is homologous to $(K,\partial K)$. 
On the other hand, as the $L_j'$ are entirely in $A$, the chain $\sum_j(L_j',\partial L_j')$ is clearly nullhomologous
in $H_p(M',A)$.  Hence $(K,\partial K)$ is homologous to $\sum_i(L_i,\partial L_i)$ 
and as the latter forms a chain in $C_p(M,\partial M)$, the assertion has been proved.
\end{proof}

%%%%%%%%%%%%%%%%%%%%%%%%%%%%%%%%%%%%%%%%%%%%%%%%%%%%%%%%%%%%%%%%%%%%%%%%%
%																			    %
%																			    %
%  Analytic and geometric preliminaries                                         	                                           %               	
%																			    %
%																			    %
%%%%%%%%%%%%%%%%%%%%%%%%%%%%%%%%%%%%%%%%%%%%%%%%%%%%%%%%%%%%%%%%%%%%%%%%%

\section{Analytic and geometric preliminaries}
\label{sec:lifting}

We adopt the usual convention that a map is holomorphic on a closed subset of 
a complex manifold if it is holomorphic on an open neighborhood of that set. 
When talking about a homotopy of such maps, it is understood that the 
neighborhood is independent of the parameter.

We shall need the following transversality theorem.
Although it follows easily from known results, lacking a precise reference we include a proof. 

%
%
%   A TRANSVERSALITY THEOREM
%
%
\begin{theorem} \label{th:transv}
Assume that $V$ is a compact strongly pseudoconvex domain with $\cC^2$ boundary in a Stein manifold $S$,
$X$ is a complex manifold, and $M\subset X$ is a smooth submanifold.
Let $r\in \{0,1,2,\ldots,\infty\}$. Every map $f\colon V\to X$ of class $\cC^r$ which is
holomorphic in the interior $\mathring V=V\setminus \partial V$ can be approximated 
arbitrarily closely in the $\cC^r(V,X)$ topology by holomorphic maps $\tilde f:\wt V\to X$,
defined on an open neighborhood $\wt V \subset S$ of $V$ (depending on $f$), 
such that $\tilde f$ is transverse to $M$. 
\end{theorem}

\begin{proof}
By \cite[Theorem 1.2]{DF2008} we can approximate the map $f\colon V\to X$ 
arbitrarily closely in the $\cC^r(V,X)$ topology by a holomorphic map 
$f_1:\Omega \to X$ on an open Stein neighborhood $\Omega\subset S$ of $V$.  
Pick a compact $\cO(\Omega)$-convex subset $K\subset \Omega$ with $V\subset \mathring K$. 
By Kaliman and Zaidenberg \cite{KZ} we can approximate $f_1$ as
closely as desired, uniformly on $K$, by a holomorphic map $\tilde f\colon \wt V\to X$ 
on an open neighborhood $\wt V$ of $K$ such that $\tilde f$ is transverse to $M$. 
(See also Theorem 7.8.12 in \cite[p.\ 321]{FF:book}. 
In the cited sources this transversality theorem is stated for the case 
when $M$ is a complex submanifold, or a Whitney stratified complex subvariety of $X$,
but the proofs also apply to smooth submanifolds; see \cite{Trivedi}.) 
\end{proof}

\begin{remark} 
We also have the corresponding jet transversality theorem --- the map $\tilde f$ 
in Theorem \ref{th:transv} can be chosen such that its $r$-jet extension 
$j^r f\colon \wt V \to \cJ^r(\wt V,X)$ is transverse to a given 
smooth submanifold $M$ of the complex manifold $\cJ^r(\wt V,X)$ of $r$-jets 
of holomorphic maps $\wt V\to X$. If $X$ is an Oka manifold, 
then the jet transversality theorem holds for
holomorphic maps $S\to X$ from an arbitrary Stein manifold $S$ to $X$; 
see \cite{F2006} or \cite[Sec.\ 7.8]{FF:book}. We shall not need these additions.
\end{remark}

Given a function $\rho\colon X\to \R$, we shall use the notation
\[
		X_{\rho<a}=\{\x\in X: \rho(\x) <a\}, \quad X_{\rho>a}=\{\x\in X: \rho(\x) >a\},
\]
\[		X_{a<\rho<b}=\{\x\in X: a<\rho(\x) <b\}, \quad X_{\rho=a}=\{\x\in X: \rho(\x) =a\},
\]
and similarly for the weak inequalities. If $\rho$ is an exhaustion function, then 
the sets $X_{\rho\le b}$ and $X_{a\le \rho\le b}=X_{\rho\le b}\setminus X_{\rho<a}$ 
are compact for every pair $a,b\in \R$.

The following is the main result of the paper \cite{DF2010}, 
stated in the precise way suited to our present purposes. 
(The special case when $V$ is a bordered Riemann surface is treated in \cite{DF2007}.
The case when $X$ is a domain in $\C^n$ is due to Dor \cite{Dor1995}.)
This is the key analytic ingredient in the proof of our main results.

%
%
%   THE MAIN THEOREM OF DRINOVEC DRNOVSEK & FORSTNERIC, DF2010
%
%
\begin{theorem} {\rm (\cite[Theorem 1.1]{DF2010})} \label{th:lifting}
Let $X$ be a complex manifold of dimension $n>1$,  $\rho\colon X\to\R$ a smooth exhaustion function, 
and $\dist$ a distance function on $X$ inducing the manifold topology. Assume that, 
for some pair of numbers $a<b$, the restriction of $\rho$ to $X_{a \le \rho \le b}$ is a $q$-convex Morse function.
Let $V$ be a compact, smoothly bounded, strongly pseudoconvex domain in a Stein manifold $S$ of dimension $p$, 
and let $f_0 \colon V \to X$ be a holomorphic map such that $f_0(\s)<b$ for all $\s\in V$ and 
$a< f_0(\s) < b$ for all $\s\in \partial V$. 
Suppose that at least one of the following two conditions holds:
\begin{itemize}
\item[\rm (a)] $r:=n-q+1\ge 2p$;
\item[\rm (b)] $r>p$ and $\rho$ has no critical points of index $>2(n-p)$ in $X_{a<\rho<b}$.
\end{itemize}
Given a compact set $K\subset \mathring V$ and numbers $\gamma \in (a,b)$ and $\epsilon>0$, there is a 
homotopy $f_t \colon V\to X$ $(t\in[0,1])$ of holomorphic maps %, with $f_0$ the given initial map,
satisfying the following properties:
\begin{itemize}
\item[\rm (i)] $f_t(\s) < b$ for all $\s \in V$ and $t\in [0,1]$,
\item[\rm (ii)] $\gamma < f_1(\s) < b$ for all $\s\in \partial V$,
\item[\rm (iii)] $\sup_{\s\in K} \dist(f_t(\s),f_0(\s)) <\epsilon$ for all $t\in [0,1]$, and
\item[\rm (iv)] $\rho(f_t(\s)) > \rho(f_0(\s))-\epsilon$ for all $\s\in V$ and $t\in [0,1]$.  
\end{itemize}
If $2p<n$ then the map $f_1$ can be chosen an embedding, and if $2p=n$ then it can be chosen an immersion
with simple double points (normal crossings).  

If $\rho$ is Morse and $q$-convex on $X_{\rho> a}$ and either 
of Conditions (a), (b) holds on $X_{\rho > a}$, 
then $f_0$ can be approximated uniformly on compacta in $\mathring V$ by proper 
holomorphic maps $\tilde f: \mathring V\to X$ (embeddings if $2p<n$, immersions with 
normal crossings if $2p=n$).
\end{theorem}

A few comments are in order, especially since we shall use not just the result itself,
but also some of the key steps in the proof. 

Recall that the number $r=n-q+1\le n$, appearing in Conditions (a) and (b), is a pointwise 
lower bound on the number of positive eigenvalues of the Levi form of $\rho$. 

In the {\em noncritical case}, i.e., when $\rho$ has no critical values in 
$[a,b]$, Theorem \ref{th:lifting} holds under the condition that
$r>p$. More precisely, the Levi form of $\rho$ must have at least 
$p$ positive eigenvalues in directions tangential to the level sets
of $\rho$; the radial direction is irrelevant in this problem. 
The construction of a new map $f_1$ with boundary in $X_{\gamma <\rho<b}$
is achieved in finitely many steps by successively lifting small portions 
of the boundary of the image of $V$ in $X$ to higher level sets of 
$\rho$ (see Condition (ii) in Theorem \ref{th:lifting}), 
paying attention to remain within $X_{\rho<b}$ 
(Condition (i)) and not to drop very much anywhere (Condition (iv)).
Every step of the deformation is first carried out in a local chart on $X$
by pushing the image in a suitable $p$-dimensional direction, tangential to the 
level set of $\rho$, on which the Levi form of $\rho$ is strictly positive.
Special holomorphic peaking functions constructed by Dor \cite{Dor1995}
are used for this purpose. The deformation is globalized by the method of
gluing holomorphic sprays, developed in \cite{DF2007}. 

The second part of Condition (b) is used in the following way.
When trying to push the (image in $X$ of the) boundary of $V$ across a critical 
point of index $k$ of $\rho$, we must be able to ensure that 
$f(\partial V)$ (which is of real dimension $2p-1$) avoids the real 
$k$-dimensional stable manifold of the critical point.
By a general position argument this is possible if $(2p-1)+k<2n$,
which is equivalent to $k\le 2(n-p)$. 

Condition (a), that $n-q+1\ge 2p$, is equivalent to $n+q-1= 2n-(n-q+1)\le 2(n-p)$. 
Since Morse indices of a $q$-convex function are $\le n+q-1$, this implies Condition (b).

The last statement, on the existence of proper maps, follows from the first one 
by a standard inductive procedure. By combining the general position argument
and approximating sufficiently well at every step, we can also ensure to obtain in the limit 
a proper holomorphic embedding (if $2p<n$) or immersion (if $2p=n$). All this is very 
standard and well known at least since Whitney's classical work on immersions and embeddings. 
In view of \cite[Theorem 1.2]{DF2008}, the analogous result also holds 
if the initial map $f_0\colon V\to X$ is merely continuous on $V$ and holomorphic 
on $\mathring V=V\setminus\partial V$.

%
%
%  GEOMETRY OF CRITICAL POINTS OF Q-CONVEX FUNCTIONS
%

The reader will need a certain amount of familiarity with the proof of Theorem \ref{th:lifting}.
To this end,  we recall in some detail the geometry of critical points 
of $q$-convex functions. Our main source are Sections 2 and 3 in \cite{DF2010} where the reader
can find further details. (This is also available in Sections 3.9--3.10
of the monograph \cite{FF:book}.)

By Sard's lemma we may assume that $a$ and $b$ are regular values  of $\rho$. 
After a small deformation of $\rho$ we may  assume that its 
critical points in $X_{a<\rho<b}$ lie on different level sets. 
Let $\x_0$ be any critical point and $k_0$ 
its Morse index; hence $k_0\le n+q-1$. Set $s=q-1$, so $r+s=n$.
By \cite[Lemma 2.1, p.\ 9]{DF2010} (or \cite[Lemma 3.9.4, p.\ 91]{FF:book}) there exist
\begin{itemize}
\item[\rm (i)]  a holomorphic coordinate map $z=(\xi,w) \colon U\to \C^r\times \C^s=\C^n$ 
on an open neighborhood $U\subset X$ of $\x_0$, with $z(\x_0)=0$,
\item[\rm (ii)] an $\R$-linear change of coordinates 
$\psi(z)=\psi(\xi,w)=(\xi + l(w),g(w))$ on $\C^r\times\C^s$, 
\item[(iii)] integers $k\in\{0,\ldots,r\}$ and $m\in\{0,1,\ldots,2s\}$ with $k+m=k_0$, 
and
\item[(iv)] a quadratic $q$-convex function $\C^n\cong \C^r\times \R^{2s}\to\R$ of the form 
\begin{equation}
\label{normal-form}
	\wt\rho(\zeta,u) = - |x'|^2  -|u'|^2 + |x''|^2 + 
					|u''|^2 + \sum_{j=1}^r \lambda_j y_j^2, 
\end{equation}	
where $\zeta=(\zeta',\zeta'') \in \C^k\times \C^{r-k}$, $\zeta'=x'+\I y'\in \C^k$, 
$\zeta''=x''+\I y''\in \C^{r-k}$, $\lambda_j>1$ for $j=1,\ldots,k$, 
$\lambda_j\ge 1$ for $j=k+1,\ldots,r$, and $u=(u',u'') \in \R^m\times \R^{2s-m}$,
\end{itemize} 
such that, setting 
\[
	\phi(\x)=  \psi(z(\x)) = (\zeta(\x),u(\x)) \in \C^n,\quad \x\in U,
\]
we have
\begin{equation}\label{eq:rho}
	\rho(\x) = \rho(\x_0) + \wt\rho(\phi(\x)) + o(|\phi(\x)|^2),\quad \x\in U. 
\end{equation}

The function $\wt\rho$ (\ref{normal-form}) has a critical point of Morse index $k_0$ at 
the origin (due to the term $- |x'|^2  -|u'|^2$) and no other critical points.
For every fixed $u\in\R^{2s}$ the function $\C^r\ni \zeta \mapsto \wt\rho(\zeta,u)$
is strongly plurisubharmonic, so $\wt \rho$ is $q$-convex on $\C^r\times\C^s=\C^n$. 
Since the $\R$-linear map $\psi$ preserves the foliation $u=const$ and 
is $\C$-linear on each leaf $\C^r\times\{u\}$, 
the function $\wt\rho\circ\psi$ is also $q$-convex on $\C^n$. 

By a small deformation of $\rho$ in a neighborhood of the critical point 
$\x_0$ we may assume that the remainder term in (\ref{eq:rho}) vanishes 
(cf.\ \cite[Lemma 2.1, p.\ 9]{DF2010}); a critical point with this property 
is said to be {\em nice}. Hence we may work in the sequel with $q$-convex functions with nice
critical points.

Assume without loss of generality that $\rho(\x_0)=0$. By shrinking $U$ around 
the point $\x_0$, we may assume that $\phi$ maps $U$ into a polydisc
$P\subset \C^r\times\R^{2s}$ around the origin. Write 
\[
		\wt\rho(x+\mathrm{i} y,u) = -|x'|^2 - |u'|^2 + Q(y,x'',u'')
\]
where $\zeta=x+\mathrm{i}y\in\C^r$, $u\in\R^{2s}$ and
\begin{equation}\label{eq:Q}
	Q(y,x'',u'')= \sum_{j=1}^r \lambda_j y_j^2 + |x''|^2 + |u''|^2. 
\end{equation}
Pick a number $c_0\in (0,1)$ small enough such that $\rho$ has no critical
points other than $\x_0$ in the layer $X_{-c_0\le \rho\le 3c_0} \Subset X_{a<\rho < b}$,
and we have
\[
		\{(x+\mathrm{i} y,u)\in\C^r\times \R^{2s} \colon 
			 |x'|^2 + |u'|^2 \le c_0,\ Q(y,x'',u'') \le 4c_0\} 
			 \subset P.
\]
Consider the set 
\begin{equation}\label{eq:wtE}
	\wt E=\{(x+\mathrm{i} y,u)\in\C^r\times \R^{2s} \colon |x'|^2 + |u'|^2 \le c_0,\ 
	 			 y=0,\ x''=0,\ u''=0\}.
\end{equation}
Its preimage $E=\phi^{-1}(\wt E)\subset U \subset X$ is an embedded real analytic disc 
of dimension $k+m=k_0$ (the Morse index of $\rho$ at $\x_0$) which is attached
to the domain $X_{\rho\le -c_0}$ along the sphere $\S^{k_0-1} \cong \partial E\subset X_{\rho = -c_0}$. 
(In the metric on $U$ inherited from the Euclidean metric in $\C^n$ by the coordinate map 
$\phi:U\to P\subset \C^n$, $E$ is the stable manifold of $\x_0$ for the gradient flow of $\rho$.)

Let $\lambda=\min \{\lambda_1,\ldots,\lambda_k\}>1$. Pick a number $t_0$ with 
$0< t_0 < (1-\frac{1}{\lambda})^2c_0$.

By \cite[Lemma 6.7, p.\ 178]{F2003} (or \cite[Lemma 3.10.1, p.\ 92]{FF:book})
there exists a smooth convex increasing function $h\colon \R \to [0,+\infty)$ 
enjoying  the following properties:
\begin{itemize}
\item[(i)]   $h(t)=0$ for $t\le t_0$, 
\item[(ii)]  for $t\ge c_0$ we have $h(t)=t - t_1$, where $t_1=c_0 - h(c_0) \in (t_0,c_0)$,
\item[(iii)] for $t_0\le t\le c_0$ we have $t-t_1 \le h(t) \le t - t_0$, and
\item[(iv)]  for all $t\in\R$ we have $0\le \dot h(t) \le 1$ and $2t\ddot h(t) + \dot h(t) < \lambda$.
\end{itemize}
The graph of $h$ is shown on Fig.\ 2 in \cite[p.\ 11]{DF2010} and on Fig.\ 3.4 in \cite[p.\ 93]{F2003}.

With $Q$ as in (\ref{eq:Q}), we consider the smooth function 
$\wt \tau\colon \C^n\cong \C^r\times \R^{2s} \to\R$ given by 
\begin{equation}\label{eq:wttau}
		\wt\tau(\zeta,u)= - h\bigl( |x'|^2 +|u'|^2\bigr) +  Q(y,x'',u'').
\end{equation}
Using the properties of $h$, it is easily verified that 
the function $\C^r\ni\zeta \mapsto \wt\tau(\zeta,u)$ 
is strongly plurisubharmonic for every fixed $u\in\R^{2s}$, so
$\wt\tau$ is $q$-convex on $\C^n$. By the same argument as above
(due to the special form of the $\R$-linear map $\psi$), 
the composition $\wt\tau\circ\psi$ has the same property and 
hence is also $q$-convex on $\C^n$. Furthermore, $0$ is a critical value 
with critical locus $\{|x'|^2+|u'|^2\le t_0,\ x''=0,\ y=0,\ u''=0\}$, 
and $\wt \tau$ has no critical values in $(0,+\infty)$. 
(See \cite[Lemma 3.1, p.\ 10]{DF2010} or \cite[Lemma 3.10.1, p.\ 92]{FF:book}
for the details.) 

Let $\phi=\psi\circ z \colon U\to\C^n$ be as above.
Define the  function $\tau \colon X_{\rho\le 3c_0} \to\R$ by 
\begin{equation}\label{eq:tau}	
	\tau= \wt\tau\circ \phi =\wt\tau\circ \psi \circ z \ \; \text{on}\ U\cap X_{\rho\le 3c_0}, 
	\quad  \tau = \rho + t_1 \ \text{on}\  X_{\rho\le 3c_0}\setminus U. 
\end{equation}
It is easily seen that $\tau$ is well defined and enjoys the following properties
(the details can be found in \cite[Lemma 3.1, p.\ 10]{DF2010} or 
\cite[Lemma 3.10.1, p.\ 92]{FF:book}; for the strongly pseudoconvex case 
see also \cite[Lemma 6.7, p.\ 178]{F2003}): 
\begin{itemize}
\item[($\alpha$)] $E\cup X_{\rho\le -c_0}  \subset X_{\tau\le 0} \subset E\cup X_{\rho\le -t_0}$, 
\item[($\beta$)]  $X_{\rho \le c_0} \subset X_{\tau \le 2c_0} \subset X_{\rho< 3c_0}$,
\item[($\gamma$)] $\tau$ is $q$-convex on the set  $X_{-c_0\le \rho\le 3c_0}$, and  
\item[($\delta$)] $\tau$ has no critical values in the interval $(0,2c_0]$. (The level set $X_{\tau=0}$ is critical.) 
\end{itemize}
The critical locus of $\tau$ equals $\{|x'|^2+|u'|^2\le t_0,\ x''=0,\ y=0,\ u''=0\} \subset E$, 
and $\tau$ vanishes on this set. For every $c\in (0,2c_0]$ the sublevel set 
\begin{equation}\label{eq:Omega-c}
		\Omega_c=\{\x\in X_{\rho<3c_0}: \tau(\x) \le c\} \Subset X_{\rho<3c_0}
\end{equation} 
is a smoothly bounded $q$-complete domain. As $\tau$ has no critical values in 
$(0,2c_0]$, these domains are diffeomorphic to each other.
If $c>0$ is chosen sufficiently small such that $c-t_1<0$
(the number $t_1>0$ was defined in property (ii) of $h$), then $\Omega_c$ is obtained
by attaching to the subcritical sublevel set $X_{\rho\le c-t_1}$ 
a handle of index $k_0$ corresponding to the critical point of $\rho$ at $\x_0$.
Hence every domain $\Omega_c$ for $c\in (0,2c_0]$ is diffeomorphic to 
the sublevel set $ X_{\rho \le c}$ of $\rho$ which contains the critical point $\x_0$
in its interior (since $c>0$).  The inclusion $E\cup X_{\rho\le c-t_1} \hookrightarrow \Omega_c$
is a homotopy equivalence since  $E\cup X_{\rho\le c-t_1}$ is a strong deformation retract of $\Omega_c$.
From the second equation in (\ref{eq:tau}) we also see that
\[
		\Omega_c \setminus U = X_{\rho \le c-t_1}\setminus U.
\]
The sets $\Omega_c$ are illustrated on Fig.\ \ref{tau}
which is reproduced from Fig.\ 1 in \cite[p.\ 10]{DF2010} 
and also Fig.\ 3.5 in \cite[p.\ 94]{FF:book}.

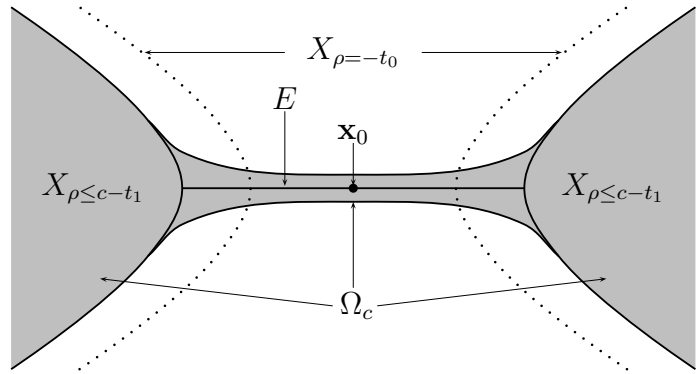
\begin{figure}[ht]
\psset{unit=0.6cm, xunit=1.5, linewidth=0.7pt} %% 
\begin{pspicture}(-4.5,-4.3)(4.5,4.3)

%
%
% Geometry 
%
%
\pscustom[fillstyle=solid,fillcolor=lightgray,linestyle=none]  % central part 
{
\pscurve(-3,-1.5)(-2.5,-0.8)(0,-0.3)(2.5,-0.8)(3,-1.5) 
\psline[linestyle=dashed,linewidth=0.2pt](3,-1.5)(3,1.5)
\pscurve[liftpen=1](3,1.5)(2.5,0.8)(0,0.3)(-2.5,0.8)(-3,1.5)
\psline[linestyle=dashed,linewidth=0.2pt](-3,1.5)(-3,-1.5)
}

\pscustom[fillstyle=solid,fillcolor=lightgray]
{\pscurve[liftpen=1](5,4)(3,1.5)(2.5,0)(3,-1.5)(5,-4)          % right hyperbola 
}

\pscustom[fillstyle=solid,fillcolor=lightgray]
{
\pscurve[liftpen=1](-5,4)(-3,1.5)(-2.5,0)(-3,-1.5)(-5,-4)      % left hyperbola 
}

%
%  The core disc and the circle
%
%
\psline(-2.5,0)(2.5,0)                                        % the core disc
\psecurve(5,4)(3,1.5)(2.5,0.8)(0,0.3)(-2.5,0.8)(-3,1.5)(-5,4) % upper boundary of L
\psecurve(5,-4)(3,-1.5)(2.5,-0.8)(0,-0.3)(-2.5,-0.8)(-3,-1.5)(-5,-4)  
																															% lower boundary of L

%
%
%  The level $h=c_0$
%
\pscurve[linestyle=dotted,linewidth=1pt](4,4)(2,1.5)(1.5,0)(2,-1.5)(4,-4)             \pscurve[linestyle=dotted,linewidth=1pt](-4,4)(-2,1.5)(-1.5,0)(-2,-1.5)(-4,-4)            
% left hyperbola 

\rput(0,3){$X_{\rho=-t_0}$}
\psline[linewidth=0.2pt]{->}(1,3)(3.05,3)
\psline[linewidth=0.2pt]{->}(-1,3)(-3.05,3)

%
%
%   NOTATION
%

\psline[linewidth=0.2pt]{->}(0,-2.2)(0,-0.3)
\psline[linewidth=0.2pt]{<-}(-3.7,-2)(-0.3,-2.6)
\psline[linewidth=0.2pt]{->}(0.35,-2.6)(3.7,-2)
\rput(0.05,-2.6){$\Omega_c$}

\psline[linewidth=0.2pt]{->}(-1,1.7)(-1,0.05)
\rput(-1,2){$E$}

\psdot(0,0)
\rput(0,1.2){$\x_0$}
\psline[linewidth=0.2pt]{->}(0,1)(0,0.05)

\rput(3.8,0){$X_{\rho \le c-t_1}$}
\rput(-3.8,0){$X_{\rho \le c-t_1}$}

\end{pspicture}
\caption{The set $\Omega_c=X_{\tau \le c}$}
\label{tau}
\end{figure}

Another observation will be crucial for our purposes. 
On the set $U\cap \{|x'|^2 +|u'|^2\le t_0\}$ we have 
$h(x'(\x),u'(\x))=0$, and hence $\tau(\x)=Q(y(\x),x''(\x),u''(\x))$. Therefore
\[
	 \Omega_c\cap U\cap \{|x'|^2 +|u'|^2\le t_0\} = U\cap \{|x'|^2 +|u'|^2\le t_0,\ Q\le c\}
\]
is a tube around $E$ which decreases to the disc $E\cap \{|x'|^2 +|u'|^2\le t_0\}$
as $c\searrow0$. Note that  $\Sigma :=\{z\in \C^n\colon y=0,\ x''=0,\ u''=0\}$
is the $\R$-linear subspace containing the disc $\wt E$ (\ref{eq:wtE}).
If $\Lambda\subset \C^n$ is any $\R$-linear subspace of dimension
$\dim_\R \Lambda=2n-\dim_\R \Sigma = 2n-k_0$ intersecting 
$\Sigma$ transversely at the origin, then for small enough $c\in (0,2c_0)$
the intersection $\Lambda\cap \{Q \le c\}$ is an ellipsoid in $\Lambda$.
If $2n-k_0=2p$ is even, we can pick $\Lambda$ such that
$\wt \Lambda=\psi^{-1}(\Lambda) \subset\C^n$ is a $p$-dimensional 
$\C$-linear subspace of $\C^n$. The set  
\[ 		
		V_c := \Lambda \cap \{Q \le c\} 
\] 
for small $c>0$ is an ellipsoid,  and $\wt V_c:=\psi^{-1}(V_c)$ is an ellipsoid in $\wt \Lambda \cong\C^p$.  
The set 
\begin{equation}\label{eq:Zc}
		Z_c := \{\x\in\Omega_c \cap U \colon \phi(\x) \in V_c\}
		    = \{\x\in\Omega_c \cap U \colon z(\x) \in \wt V_c \}
\end{equation}
is then a closed complex submanifold of $\Omega_c=X_{\tau\le c}$, with boundary
$\partial Z_c \subset \partial \Omega_c = X_{\tau= c}$, which is 
biholomorphic to an ellipsoid in $\wt\Lambda \cong \C^p$ 
via the holomorphic coordinate map $z \colon U\to\C^n$. 
The submanifold $Z_c$ is a holomorphic cross-section of the handle that was attached to
the sublevel set $X_{\rho\le c-t_1}$ in order to get the domain $\Omega_c$.

It is now possible to explain how to lift the boundary of the image variety $f(V)\subset X$ 
across a critical level of the $q$-convex function $\rho$. We shall do this in 
the following section in the context of proving Theorem \ref{th:technical}.

%%%%%%%%%%%%%%%%%%%%%%%%%%%%%%%%%%%%%%%%%%%%%%%%%%%%%%%%%%%%%%%%%%%%%%%%%
%																			    %
%																			    %
%  Proofs                                                                                           	                                           %               			
%																			    %
%																			    %
%%%%%%%%%%%%%%%%%%%%%%%%%%%%%%%%%%%%%%%%%%%%%%%%%%%%%%%%%%%%%%%%%%%%%%%%%

\section{Proof of Theorems \ref{th:Hodge1} and \ref{th:Hodge2}}  \label{sec:proofs}
We first explain how Theorem \ref{th:lifting} and its proof
(see Sect.\ \ref{sec:lifting} above) imply the following result.
Theorems \ref{th:Hodge1} and \ref{th:Hodge2} will follow easily from Theorem \ref{th:technical}. 

%
%
%   THE MAIN TECHNICAL RESULT  
%
%
\begin{theorem}\label{th:technical}
(Hypotheses as in Theorem \ref{th:lifting}.)  Assume that for some 
$\gamma_0 \in (a,b)$ the function $\rho$ has no critical values on $[a,\gamma_0]$,
and every element of the homology group $H_{2p}(X_{\rho<\gamma_0},X_{a<\rho<\gamma_0};\Z)$ 
can  be represented by a finite analytic cycle $\sum_i n_i Z_i$,  where every
$Z_i$ is a holomorphic image of a strongly pseudoconvex Stein domain $V_i$. 
Then for any $\gamma_1 \in (\gamma_0,b)$,  each element of the homology group 
$H_{2p}(X_{\rho<b},X_{\gamma_1 <\rho<b};\Z)$ is also represented by a finite analytic cycle of 
the same form. If in addition the elements of $H_{2p}(X_{\rho<\gamma_0},X_{a<\rho<\gamma_0};\Z)$
are representable by cycles as above in which every domain $V_i$ is the ball $\B^p\subset\C^p$,
then the same is true for  the group $H_{2p}(X_{\rho<b},X_{\gamma_1 <\rho<b};\Z)$.
\end{theorem}

\begin{proof}
Let us first consider the {\em noncritical case} when $\rho$ has no critical values in 
$[a,b]$. As there is no change of topology of the sublevel sets 
of $\rho$, the relative homology groups 
$H_{j}(X_{\rho\le t},X_{\rho=t};\Z)$ are isomorphic to each other for all
$t\in [a,b]$ and all $j\in \Z_+$. Furthermore, for any pair of numbers
$\alpha,\beta$ with $a\le \alpha <\beta \le b$, the set $X_{\alpha\le \rho\le \beta}$
is an interior collar around the boundary $X_{\rho=\beta}$ of the manifold 
with boundary $X_{\rho\le \beta}$, and also an exterior collar of 
$X_{\rho=\alpha} = \partial (X_{\rho\le \alpha})$. 
By Lemma \ref{lem:collar} and excision we thus get an isomorphism
\begin{equation}\label{eq:iso-k}
		H_{j}(X_{\rho<\gamma_0},X_{a<\rho<\gamma_0};\Z) 
		\cong H_{j}(X_{\rho<b},X_{\gamma_1 <\rho<b};\Z)
\end{equation}
for every $\gamma_1 \in(\gamma_0,b)$ and $j\in\Z_+$. 
Assume that $V$ is a strongly pseudoconvex Stein domain of dimension $p$
and $f_0:V\to X$ is a holomorphic map with range in $X_{\rho<\gamma_0}$
such that $f_0(\partial V)\subset  X_{a<\rho<\gamma_0}$;
then $f_0(V)$ represents an element of the group 
$H_{2p}(X_{\rho<\gamma_0},X_{a<\rho<\gamma_0};\Z)$.
Assuming that $p<n-q+1$, Theorem \ref{th:lifting}  tells us that
$f_0$ is homotopic to another holomorphic map $f_1:V\to X$ whose image
$f_1(V)$ represents the same element in
$H_{2p}(X_{\rho<b},X_{\gamma_1<\rho<b};\Z)$ under the isomorphism (\ref{eq:iso-k}).   
This holds in particular if $2p=n-q+1$, the case of interest to us.

It follows that we can lift an analytic $p$-cycle 
$\sum n_i f_i(V_i)$ in $(X_{\rho<\gamma_0},X_{a<\rho<\gamma_0})$, with $p<n-q+1$, to an analytic $p$-cycle
$\sum n_i \tilde f_i(V_i)$ in $(X_{\rho<b},X_{\gamma_1<\rho<b})$ such that these two cycles 
represent the same homology class under the isomorphism (\ref{eq:iso-k})
with $j=2p$. In the new cycle we use the same domains $V_i$ and weights $n_i$,
only the maps change.

This completes the analysis of the noncritical case.

Assume now that  $\rho$ has critical points in $X_{a\le \rho \le b}$.
We may assume that $a$ and $b$ are regular values, so $\rho$ (being Morse) has 
only finitely many critical points in $X_{a <\rho < b}$.
By a small deformation of $\rho$ we may assume that these points lie on
different level sets of $\rho$ 
and each of them is nice, in the sense that $\rho$ can be represented in the normal form 
(\ref{eq:rho}) without the remainder term. 
It suffices to prove Theorem \ref{th:technical} in the case when $\rho$ has 
only one critical point in $X_{a<\rho<b}$; the general case then follows by a finite induction.

Thus, let $\x_0\in X_{a<\rho <b}$ be the unique (nice) critical point of $\rho$ in 
$X_{a\le \rho\le b}$. We may assume that $\rho(\x_0)=0$. 
By the inductive hypothesis there exists a number $\gamma_0\in (a,0)$ such that
every homology class in $H_{2p}(X_{\rho<\gamma_0},X_{a<\rho<\gamma_0};\Z)$ is represented by an
analytic  cycle consisting of images of strongly pseudoconvex Stein domains. 

We shall use the notation from the previous section;
this pertains in particular to the positive numbers $c_0,t_0,t_1>0$, 
the coordinate map $\phi=\psi\circ z:U\to P\subset \C^n$ on a neighborhood $U\subset X$ of $\x_0$,
the core $E=\phi^{-1}(\wt E)$ of the handle (\ref{eq:wtE}) at $\x_0$,
the function $\tau$ (\ref{eq:tau}), its sublevel sets $\Omega_c=X_{\tau\le c}$ 
(\ref{eq:Omega-c}), and the ellipsoids $Z_c \subset \Omega_c$ (\ref{eq:Zc}).
We may assume that $\gamma_0 < -t_0<0$ and $0<3c_0<b$. 

It suffices to deal separately with each of the subvarieties in the given cycle. 
Let $f_0(V) \subset X_{\rho<\gamma_0}$ be such a subvariety, with $f_0(\partial V)\subset X_{a<\rho<\gamma_0}$.
By the noncritical case, applied on $[a,-t_0/2]$, we can represent the homology 
class $[f_0(V)]\in H_{2p}(X_{\rho<\gamma_0},X_{a<\rho<\gamma_0};\Z)$ by a subvariety 
$\tilde f_0(V) \subset X_{\rho<-t_0/2}$, with 
$\tilde f_0(\partial V) \subset X_{-t_0<\rho<-t_0/2}$, such that 
\[
	[\tilde f_0(V)] \in H_{2p}(X_{\rho< -t_0/2},X_{-t_0<\rho<-t_0/2};\Z)\cong  
	 H_{2p}(X_{\rho<\gamma_0},X_{a<\rho<\gamma_0};\Z).
\]
As $2p\le n-q+1$, the general position argument (cf.\ Theorem \ref{th:transv})
allows us to deform $\tilde f_0$ slightly to ensure that $\tilde f_0(\partial V)\cap E=\emptyset$,
so we have $\tilde f_0(\partial V) \subset X_{-t_0<\rho < -t_0/2} \setminus E$.

Recall that $\{\tau=0\}\cap  X_{\rho\ge -t_0} = E\cap X_{\rho\ge -t_0}$
(see property ($\alpha$) just below (\ref{eq:tau})).
Hence, for $c>0$ small enough, we have $\tilde f_0(\partial V)\subset X_{\tau>c}$,
so $\tilde f_0(V)$ also determines a homology class in 
$H_{2p}(X_{\tau < c_1},X_{c<\tau < c_1};\Z)$ for some constant $c_1$ satisfying $c<c_1<2c_0$.
The domain $X_{\tau\le c_1}$ is diffeomorphic to $X_{\rho \le c_1}$, i.e., it is
the manifold obtained  by attaching to $X_{\rho\le -t_0}$ a handle of index $k_0$,
representing the change of topology at the critical point $\x_0$. The homology class 
$[\tilde f_0(V)] \in H_{2p}(X_{\tau < c_1},X_{c<\tau < c_1};\Z)$ is simply the 
prolongation of the same class in $H_{2p}(X_{\rho<-t_0/2},X_{-t_0<\rho<-t_0/2};\Z)$
in the sense of Proposition \ref{cor:change}. 

By using the noncritical case with the function $\tau$ (\ref{eq:tau}),
we lift the boundary of the variety $\tilde f_0(V)$ into the layer  $X_{c_0<\rho<3c_0}$ above the 
critical level $\rho(\x_0)=0$; see conditions ($\alpha$) and ($\beta$) following (\ref{eq:tau}).
Next, we use the noncritical case, this time again with the function $\rho$,
to lift the boundary of the variety, obtained in the previous substep,
closer towards the level set $X_{\rho=b}$. The result of these two modifications 
is that we replaced $\tilde f_0(V)$ by a homologous analytic cycle
$f_1(V) \subset X_{\rho<b}$ satisfying $f_1(\partial V)\subset X_{\gamma_1 <\rho <b}$,
where $\gamma_1$ is any number with $3c_0< \gamma_1 <b$.

This shows that all analytic $p$-cycles with $2p\le n-q+1$, coming from below the critical level of $\rho$ 
at $\x_0$, survive the passage of the critical level $\rho=\rho(\x_0)$ with the same
normalizing strongly pseudoconvex domains.

In the subcritical case $2p<n-q+1$ this completes the proof since
the relative homology group $H_{2p}(X_{\rho<t},X_{\rho=t};\Z)$ does not change 
as $t$ passes the value $\rho(\x_0)$. The same is true if $2p=n-q+1$
and the Morse index of $\x_0$ is $<n+q-1$.

If $2p=n-q+1$, a critical point $\x_0$ of maximal Morse index
$n+q-1$ of $\rho$ may give birth to an additional generator of the relative homology 
(cf.\ Proposition \ref{cor:change}). As we have seen in Sect.\ \ref{sec:lifting}, 
this new generator can be represented by a properly embedded
complex ellipsoid $Z_c \subset X_{\tau\le c}$ with $\partial Z_c\subset X_{\tau=c}$ (\ref{eq:Zc}).

We now explain how to replace this ellipsoid $Z_c$ by a ball in the subsequent construction.

Pick a closed domain $B\subset Z_c$, biholomorphic to the closed ball 
$\overline\B^p\subset\C^p$, which contains the (unique) intersection point 
$Z_c\cap E$ in its relative interior. We have $B=g(\bar \B^p)$ for a holomorphic
embedding $g: \bar \B^p\to X$. There is a constant $c'\in (0,c)$ 
such that $Z_c\cap X_{\tau \le c'}$ is contained in the relative interior of $B$. 
Since the function $\tau$  (\ref{eq:tau}) has no critical values in the interval $(0,2c_0]$
by property ($\delta$), we see by the same argument as above 
that $Z_c$ and $B$ determine the same element in the relative homology
group $H_{2p}(X_{\tau<2c_0},X_{c'<\tau<2c_0};\Z)$.
Applying the noncritical case, first with $\tau$ (to push the boundary of 
the ball $B$ into a $\rho$-layer above the critical level) and subsequently with $\rho$,
we can replace this new component by another holomorphic immersion
(embedding if $2p<n$) of  $(\bar \B^p,\partial \B^p)$ into 
$(X_{\rho<b},X_{\gamma_1<\rho<b})$.

This completes the proof of Theorem \ref{th:technical}.
\end{proof}

%
%
%   PROOF OF THEOREM 2
%
%
\begin{proof}[Proof of Theorem \ref{th:Hodge2}]
Choose a $q$-convex Morse function $\rho$ on a neighborhood 
of $M$ in $X$ such that $M=\{\rho\le 0\}$ and $d\rho\ne 0$ on $\partial M$.
Pick numbers $a<\gamma_0< \gamma_1 <b=0$ such that
$\gamma_0<\min_M\rho < \gamma_1$, and $\gamma_1$ 
is close enough to $0$ so that $\rho$  has no critical values on $[\gamma_1,0]$.
Then $X_{\gamma_1\le \rho\le 0}$ is a collar around $\partial M=X_{\rho=0}$,
and by choosing $\gamma_1$ sufficiently close to $0$ we may assume that it is 
contained in the given collar $A$.  Since $X_{\rho<\gamma_0}=\emptyset$, the hypotheses of  
Theorem \ref{th:technical} are trivially satisfied.  The result now follows directly from 
Theorem \ref{th:technical}.
\end{proof}

%
%
%   PROOF OF THEOREM 1
%
%
\begin{proof}[Proof of Theorem \ref{th:Hodge1}]
Let $M$ be a compact $q$-complete domain in a complex manifold $X$.
Let $A\subset X$ be an interior collar and $B\subset X$ be an exterior collar around $\partial M$,
and set $\wt M=M\cup B$. For any $j\in\N$ we have natural isomorphisms 
\[
	H_j(M,\partial M;\Z)\cong H_j(M,A;\Z) \cong H_j(\wt M,B;\Z).
\]
Let $2p=n-q+1$. By Theorem \ref{th:Hodge2} every homology class $z\in H_{2p}(M,\partial M;\Z)$
is represented by an analytic cycle $Z=\sum_i n_i Z_i$,
where each $Z_i=f_i(\bar \B^p)$ is a holomorphic image of the ball 
$\bar \B^p \subset\C^p$ and  $\partial Z_i = f_i(\partial \B^p) \subset A$.
By the noncritical case of Theorem \ref{th:technical} we can replace $Z$
by a homologous analytic cycle $\wt Z=\sum_i n_i \wt Z_i$ in $(\wt M,B)$,
where $\wt Z_i  =\tilde f_i(\bar \B^p)$ for some holomorphic map 
$\tilde f_i\colon \bar \B^p \to \wt M$ with $\tilde f_i(\partial \B^p)\subset B \setminus \partial M$. 
In view of Theorem \ref{th:transv} we can assume that each $\tilde f_i$ is transverse
to $\partial M$. By Proposition \ref{prop:extcollar}, 
a suitable integral linear combination of the intersections of the components of 
$\wt Z$ with $M$ gives an analytic cycle in $(M,\partial M)$ 
which determines the homology class $z\in H_{2p}(M,\partial M;\Z)$. 
The cycle in $(M,\partial M)$ obtained in this way is of the form 
$\sum_{i,j} n_{i,j} \tilde f_i(W_{i,j})$, where $W_{i,j}$ are smoothly bounded 
connected domains in $\B^p$.
\end{proof}

%%%%%%%%%%%%%%%%%%%%%%%%%%%%%%%%%%%%%%%%%%%%%%%%%%%%%%%%%%%%%%%%%%%%%%%%%
%																			    %
%																			    %
%  Proof of Theorem \ref{th:Hodge3}	                                        	                                           %               			
%																			    %	
%%%%%%%%%%%%%%%%%%%%%%%%%%%%%%%%%%%%%%%%%%%%%%%%%%%%%%%%%%%%%%%%%%%%%%%%%

\def\lone{\lim{}^{\!\!1\,}}
\newcommand{\limone}[1]{\lim_{#1}{}^{\!\!1\,}}

\section{Proof of Theorem \ref{th:Hodge3}} \label{sec:openMfd}

Let $X$ be a $q$-complete manifold without boundary, and let $\rho\colon X\to\R$ be a $q$-convex Morse
exhaustion function. We choose an exhaustion $M_1\subset M_2\subset \dots\subset\cup_{j=1}^\infty M_j=X$
where each $M_j$ is a regular sublevel set of $\rho$ and there is at most one critical point of $\rho$ in
each difference $\mathring{M_j}\setminus M_{j-1}$. Recall the definition $\cH^k(X;G)=\lim_j H^k(M_j;G)$ and
note that this definition does not depend on the particular choice of $\rho$ or the sublevel sets. The
easiest way to see this is to define $\cH^k(X;G)$, in analogy with Atiyah and Hirzebruch \cite{AH1961},
as the inverse limit of groups $H^k(M;G)$ where $M$ ranges over all compact subdomains of $X$. 
A particular exhaustion gives rise to a countable cofinal subset; hence the two inverse limits are isomorphic.

From now on, homology and cohomology will be taken with integer coefficients; we drop coefficients from the notation.

Let $X$ have complex dimension $n$ and let $Z$ be a closed analytic subspace of complex dimension $p=n-k=(n-q+1)/2$.
We assume that the singular subspace $Z_s$ of $Z$ is either empty (i.e. $Z$ is a submanifold) or discrete.
(These are the only cases that we need to consider for our purposes.)
We define the cohomology class dual to $Z$ as follows.
(See Atiyah and Hirzebruch \cite{AH1961}, Section 5, and also Douady \cite{Douady1961}, Section V.A.) As $Z'=Z-Z_s$
is a closed submanifold of $X'=X-Z_s$, it admits a (smooth) tubular neighborhood, say $W'$, in $X'$. There
is a Thom class in $H^{2k}(W',W'-Z')$ induced by the canonical orientation of the complex normal bundle,
and consequently, by excision, the corresponding
class in $H^{2k}(X',X'-Z')$. By restriction, we obtain a class in $H^{2k}(X')$ and hence a class $\zeta'\in\cH^{2k}(X')$.
By Lemma 5.3 of \cite{AH1961}, the natural `restriction' morphism $\cH^{2k}(X)\to\cH^{2k}(X')$ is an isomorphism
and we get an element $\xi\in\cH^{2k}(X)$. The element $\xi$ will be referred to as the cohomology class dual to $Z$.
More generally,
let $\setof{Z_r\,\vert\,r}$ be a countable, locally finite collection of closed analytic subspaces of $X$ whose
singular sets are either empty or discrete. Let $\xi_r\in\cH^{2k}(X)$ denote the cohomology class dual to $Z_r$
and let $n_r$ be integers. Local finiteness ensures that the possibly infinite sum $\sum_rn_r\xi_r$ makes sense
in $\cH^{2k}(X)$. To make this precise, denote the canonical projections
\[	P^i\colon \cH^{2k}(X)\to H^{2k}(M_i).		\]
By local finiteness, $P^i(\xi_r)$ is nontrivial for at most finitely many $r$, and hence $\sum_rn_rP^i(\xi_r)$
is an honest element of $H^{2k}(M_i)$. The sequence $\{\sum_rn_r P^i(\xi_r)\,\vert\,i\}$ then defines an element
of the inverse limit $\lim_iH^{2k}(M_i)=\cH^{2k}(X)$ which we denote by $\sum_rn_r\xi_r$.

We propose the following definition which is a slight generalization of the definition in \cite[Section 5.D]{AH1961}.

\begin{defn}
The cohomology class $\sum_rn_r\xi_r\in\cH^{2k}(X)$ described above is complex analytic, and it is dual to the
complex analytic cycle $\sum_rn_rZ_r$.
\end{defn}

The situation for a compact domain $M$ with boundary is as follows. Let $Z$ be an embedded 
complex submanifold of $M$ with smooth boundary, and assume that $Z$ meets $\partial M$ 
transversely in $\partial Z$. As above, we assume that $Z$ has complex dimension $p=n-k$. 
Thus, the pair $(Z,\partial Z)$ defines a homology class  $[Z,\partial Z]$ in $H_{2n-2k}(M,\partial M)$.
Then the cohomology class $x\in H^{2k}(M)$ that corresponds to $[Z,\partial Z]$ under 
Poincar\'e duality is exactly the restriction of the Thom class of $Z$ in $M$. 
Precisely, let $W$ be a tubular neighborhood of $Z$ in $M$.
By the transversality assumption, $W\cap\partial M$ is a tubular neigborhood of $\partial Z$ in $\partial M$. Thus
there is a Thom class $\tau_Z^W\in H^{2k}(W,W\setminus Z)$ induced by the canonical orientation of the normal bundle. The image of $\tau_Z^W$ under the zig-zag composite
\begin{equation*}
H^{2k}(W,W\setminus Z)\xleftarrow[\cong]{\text{excision}}H^{2k}(M,M\setminus Z)\xrightarrow{\text{restriction}}H^{2k}(M)
\end{equation*}
equals $x$. (See Bredon \cite{Bredon}, page 371.)

\begin{proof}[Proof of Theorem \ref{th:Hodge3}]
Let $\setof{M_i}$ be an exhaustion of $X$ as above. Take $x\in\cH^{2k}(X)$, let $x_i=P^i(x)\in H^{2k}(M_i)$ be the
`restrictions' of $x$ to $M_i$, and let $z_i\in H_{2n-2k}(X,X\setminus M_i)\cong H_{2n-2k}(M_i,\partial M_i)$
be the dual homology classes.

We summarize the consequences of Theorems \ref{th:Hodge2}, \ref{th:Hodge1}, and \ref{th:technical} 
and their proofs. Each homology group $H_{2n-2k}(X,X\setminus M_i)$ has a distinguished set of generators
\[	\{\zeta^i_r\,\vert\,1\le r\le\lambda_i\}	\]
where $0\le\lambda_{i-1}\le\lambda_{i}\le\lambda_{i-1}+1$, and
\[	\iota_i\colon H_{2n-2k}(X,X\setminus M_i)\to H_{2n-2k}(X,X\setminus M_{i-1})		\]
maps each $\zeta^i_r$, for $r\le\lambda_{i-1}$, to $\zeta^{i-1}_r$. If $\lambda_i>\lambda_{i-1}$
(i.e. when passing a critical point of the highest index), then $\iota_i(\zeta^i_{\lambda_i})=0$.
Since $\iota_i(z_i)=z_{i-1}$ for all $i$, we can express
\[	z_i=\sum_{r=1}^{\lambda_i}n^i_r\zeta^i_r	\]
where $n^i_r=n^{i-1}_r$ for all $r\le\lambda_{i-1}$. In particular, it makes sense to define
\[		n_r=\lim_{i\to\infty}n^i_r.			\]

We describe and make use of the geometric representatives of $\zeta^i_r$ constructed above. Assume that $r=\lambda_i$ and
$\zeta^i_{r}\in H_{2n-2k}(X,X\setminus M_i)$ is an additional generator. 
Then $\zeta^i_{r}=[f(V),f(\partial V)]$ where $V=\bar\B^p$ is the ball in $\C^p$ and
$f=f^i\colon(V,\partial V)\to(X,X\setminus M_i)$ is a holomorphic embedding 
(immersion with normal crossings when $q=1$).This gives rise to an inverse sequence of liftings
\[	
	f^j\colon(V,\partial V)\to(X,X \setminus M_j)		
\]
so that $[f^j(V),f^j(\partial V)]=\zeta^j_{r}$ for all $j>i$. In addition, we may arrange that 
$f^j$ and $f^{j-1}$ are arbitrarily close on the compact subset 
$\{v\in V\colon  d(v,bV)\geq 1/j \}$ of $V$. Finally, we may achieve by 
Theorem \ref{th:transv} that each $f^j$ is transverse to $\partial M_j$.

By choosing $f^j$ close enough to $f^{j-1}$ for each $j$, we can make the sequence $\setof{f^j}$ converge uniformly on compact subsets of $\mathring{V}=\B^p$
to a holomorphic embedding $\phi\colon\mathring V \to X$ that is transverse to all 
$\partial M_j$. For each $j\ge i$, the geometric intersection of the image 
$\phi(\mathring V)$ with $M_j$ yields a cycle 
$[\phi(\mathring V)\cap M_j,\phi(\mathring V)\cap\partial M_j]\in H_{2n-2k}(M_j,\partial M_j)$
homologous to that obtained by intersecting $f^j(V)$ with $M_j$. 
That in turn corresponds to $[f^j(V),f^j(\partial V)]=\zeta^j_{r}$ under
the isomorphism $H_{2n-2k}(M_j,\partial M_j)\cong H_{2n-2k}(X,X\setminus M_j)$.

We let $Z_{r}$ denote the image $\phi(\mathring{V})$.
By doing this for each additional generator we obtain our collection of submanifolds 
$\mathcal{Z}=\setof{Z_r\,\vert\,r}$.
This collection is locally finite in $X$ (this is because the new generator which may appear at any of the 
critical points of maximal index does not enter into any of the previous subdomains 
during the subsequent lifting procedure, cf.\ Theorem \ref{th:lifting} and the proof of Theorem \ref{th:technical}), 
and also that the cardinality of $\mathcal{Z}$ is precisely the number of critical points of the highest index. 
Let $\xi_r\in\cH^{2k}(X)$ be the cohomology class dual to $Z_r$.
To complete the proof we need to show that $x=\sum_rn_r\xi_r$.

To this end, fix some $r$ as above and let $W$ be a tubular neighborhood of $Z_r$ in $X$. 
By transversality, we may construct
$W$ so that each intersection $W_j=W\cap M_j$ is either empty or a tubular neighborhood of 
$Z_r\cap M_j$ in $M_j$; the latter for all large enough $j$.
Assume that $W\cap M_j$ is nonempty. Naturality of the Thom class guarantees that
$P^j(\xi_r)\in H^{2k}(M_j)$ is the Poincar\'e dual of
$[Z_r\cap M_j,Z_r\cap\partial M_j]\in H_{2n-2k}(M_j,\partial M_j)$. On the other hand,
$[Z_r\cap M_j,Z_r\cap\partial M_j]$ corresponds to $\zeta^j_r\in H_{2n-2k}(X,X\setminus M_j)$ as explained above.
Consequently, $P^j(\sum_rn_r\xi_r)$ corresponds precisely to $\sum_{r=1}^{\lambda_j}n^j_r\zeta^j_r=z_j$ for all $j$.
This means that $P^j(\sum_rn_r\xi_r)=x_j$ for all $j$ which implies $x=\sum_rn_r\xi_r$, as claimed.

This completes the proof for $q>1$. The straightforward adjustments necessary to prove the case $q=1$ will be
left to the reader.
\end{proof}

%%%%%%%%%%%%%%%%%%%%%%%%%%%%%%%%%%%%%%%%%%%%%%%%%%%%%%%%%%%%%%%%%%%%%%%%%
%																			     %
%	Equality																	     %	
%																			     %
%%%%%%%%%%%%%%%%%%%%%%%%%%%%%%%%%%%%%%%%%%%%%%%%%%%%%%%%%%%%%%%%%%%%%%%%%

\section{Equality of $\cH^k(X;G)$ and $H^k(X;G)$}\label{sec:equality}

Let $X$ be a noncompact smooth manifold without boundary. % as in Section \ref{sec:openMfd} above.
As is well-known in homological algebra (see Milnor \cite{Milnor}), we have the following short exact sequence
\begin{equation}\label{eq:chs}
	0\to\limone{i}H^{k-1}(M_i;G)\to H^k(X;G)\to\lim_i H^k(M_i;G)=\cH^k(X;G)\to 0.
\end{equation}
Here, $\lone$ denotes the first (right) derived functor of the inverse limit. For inverse limits of
abelian groups, it can be described as follows. If $\dots\to A_3\xrightarrow{p_3}A_2\xrightarrow{p_2}A_1$
is an inverse sequence of abelian groups, also called a tower of abelian groups, then the collection of the
`bonding' morphisms $p_i\colon A_i\to A_{i-1}$ gives rise to a morphism $P\colon\prod_{i=1}^\infty A_i\to\prod_{i=1}^\infty A_i$.
The kernel of $1-P$ is the inverse limit $\lim_iA_i$, and the cokernel can be taken as the definition of $\limone{i}A_i$.

We discuss two particular cases.

If $\setof{A_i}$ is a tower of finite-dimensional vector spaces over some field, then $\lone A_i$ is trivial (see Weibel \cite{Weibel},
Exercise 3.5.2). Thus if $G$ is a field, $\limone{i}H^{k-1}(M_i;G)$ vanishes and $H^k(X;G)\to\cH^k(X;G)$ is an isomorphism.

If $\setof{A_i}$ is a tower of finitely generated abelian groups, then $\limone{i}A_i$ is a
divisible abelian group of the form $\Ext(B,\Z)$ where $B$ is a countable torsion-free abelian group. It follows that if
nontrivial, it is uncountable. See Jensen \cite{Jensen}, Th\'eor\`eme 2.7 for a more precise description; we deal here
with the case of our interest where $A_i=H^k(M_i;\Z)$. First we need a lemma.

\begin{lemma}\label{lem:sixterm}
Let $0\to\setof{A_i}\to\setof{B_i}\to\setof{C_i}\to 0$ be a short exact sequence of towers of abelian groups.
Then there is a natural six-term exact sequence
\[	0\to\lim_iA_i\to\lim_iB_i\to\lim_iC_i\to\limone{i}A_i\to\limone{i}B_i\to\limone{i}C_i\to 0.	\]
\end{lemma}

The proof is an easy consequence of the above definition of $\lone$ and the snake lemma.

\begin{proposition}\label{prop:vanishing}
Let $\setof{M_i}$ be an exhaustion of $X$ as above and let $T$ denote the torsion subgroup of $H_{k-1}(X)$.
Then $\limone{i}H^{k-1}(M_i;\Z)$ is isomorphic to $\Ext(H_{k-1}(X)/T,\Z)$.
Hence if $\Ext(H_{k-1}(X)/T,\Z)=0$ then the natural morphism 
$H^{k}(X;\Z)\to\cH^{k}(X;\Z)$ is an isomorphism.
\end{proposition}

\begin{proof}
Universal coefficients yield a short exact sequence of towers
\[	
	0\to\setof{\Ext(H_{k-2}M_i,\Z)\,\vert\,i}\to\setof{H^{k-1}(M_i;\Z)\,\vert\,i}
	\to\setof{\Hom(H_{k-1}M_i,\Z)\,\vert\,i}\to 0.	
\]
As $\{\Ext(H_{k-2}M_i,\Z)\}$ is a tower of finite groups, Lemma \ref{lem:sixterm} and 
Proposition 2.3 of Jensen \cite{Jensen} yield an isomorphism $\limone{i}H^{k-1}(M_i;\Z)\cong\limone{i}\Hom(H_{k-1}M_i,\Z)$.
Set $T_i=T(H_{k-1}M_i)$ and note that $\Hom(H_{k-1}M_i,\Z)$
is naturally isomorphic to $\Hom(H_{k-1}M_i/T_i,\Z)$. As $H_{k-1}M_i/T_i$ is free,
$\limone{i}\Hom(H_{k-1}M_i,\Z)$ is isomorphic to $\Ext\big(\colim_i(H_{k-1}M_i/T_i),\Z\big)$
(see Jensen \cite[p.\ 16]{Jensen}). In turn, $\colim_i(H_{k-1}M_i/T_i)$ is isomorphic to 
%the quotient 
$H_{k-1}X/T$ where $T=T(H_{k-1}X)$.
\end{proof}

The following immediate corollary to Proposition \ref{prop:vanishing} is an  
addendum to Theorem \ref{th:Hodge3} in Sect.\ \ref{sec:intro}.

\begin{corollary}\label{cor:vanishing}
Let $X$ be an $n$-dimensional $q$-complete manifold with $n+q-1=2k$ even.
Denote by $T$ the torsion subgroup of $H_{n+q-2}(X)$.
If $\Ext(H_{n+q-2}(X)/T,\Z)=0$, then Theorem \ref{th:Hodge3}  applies to every cohomology class 
in $H^{n+q-1}(X;\Z)$. This holds in particular if $H_{n+q-2}(X;\Z)$ is the direct 
sum of a free abelian group and a torsion abelian group; for instance, if it is finitely generated.
\end{corollary}

%%%%%%%%%%%%%%%%%%%%%%%%%%%%%%%%%%%%%%%%%%%%%%%%%%%%%%%%%%%%%%%%%%%%%%%%%
%																			    %
%  Examples                                         										                    %
%																                            %
%%%%%%%%%%%%%%%%%%%%%%%%%%%%%%%%%%%%%%%%%%%%%%%%%%%%%%%%%%%%%%%%%%%%%%%%%

\section{Examples} \label{sec:examples}
In this section we give a few examples  which illustrate the scope of our results.

Recall  that, if $X$ is a complex manifold of dimension $n$ and $\rho:X\to\R$ is smooth exhaustion function
which is $q$-convex on the set $X_{\rho>c}$ for some $c\in\R$, then 
any compact complex subvariety $A$ of $X$ of dimension $\dim A\ge q$ is contained
in the sublevel set $X_{\rho\le c}$. This is an easy consequence of the maximum principle
for plurisubharmonic functions. In particular, a $q$-complete manifold does
not contain any compact complex subvarieties of dimension $\ge q$; this bound is sharp 
as shown by the following example.

\begin{example} \label{ex:projective}
The manifold $X=\CP^n\setminus\CP^{n-q}$ is $q$-complete for any pair of
integers $1\le q\le n$. More generally, if  $A\subset \CP^n$ is a closed projective manifold 
of complex codimension $q$ then $X=\CP^{n}\setminus A$ is $q$-convex,
and is  $q$-complete if $A$ is a complete intersection.
 (See Barth \cite{Barth1970} and Peternell  \cite{Peternell} for these facts.)
Such a manifold $X$ contains homologically nontrivial compact complex submanifolds
of any dimension from $1$ up to $q-1$ (for example, projective linear subspaces avoiding $A$), 
but every compact complex subvariety of dimension $\ge q$ in $\CP^n$ intersects 
$A$ by Bezout's theorem.
\end{example}

We begin by an example showing that, in some cases, top dimensional cohomology classes 
may be represented both by compact and also by noncompact analytic cycles.

\begin{example}\label{ex:complexcycles}
Assume that $X^n$ is a $q$-complete manifold for some $q>1$. A  compact  complex submanifold 
$C\subset X$ of complex  of dimension $p\in\{1,\ldots,q-1\}$ represents a homology class 
$[C]\in H_{2p}(X,\Z)$, and hence, by the Poincar\'e-Lefschetz duality, also
a cohomology class $u\in H^{2k}(X;\Z)$ with $k=n-p$.
If $2p=n-q+1$ then, by Theorem  \ref{th:Hodge3}, the class $u$ is also represented by an 
analytic cycle consisting of noncompact closed $p$-dimensional
subvarieties of $X$. The equality $2p=n-q+1$ holds for some $p\in\{1,\ldots,q-1\}$ if and only if
$n$ and $q$ are of different parity and $2\le n-q+1 \le 2q-2$; equivalently, 
\begin{equation}\label{eq:triples}
	3 \le q+1 \le n \le 3q-3,\quad p=(n-q+1)/2.
\end{equation}
For any triple of integers $(n,p,q)$ satisfying (\ref{eq:triples}) we get a nontrivial example 
by taking $X=\CP^n\setminus\CP^{n-q}$ and $C\cong \CP^p \subset X$.
In this case the class $[C]\in H_{2p}(X,\Z)\cong H_{2p}(\CP^n,\Z)\cong \Z$ is a generator
of the corresponding homology group.
\end{example}

In the remainder of the section we focus on examples where the Hodge
representation of top dimensional cohomology classes by compact analytic cycles is impossible,
but our results give representation by noncompact analytic cycles.

\begin{example} \label{ex:projective2}
Let $A\subset\CP^n$ be a complete intersection defined by $q$ independent holomorphic equations.
Then the complement $X=\CP^n\setminus A$ is $q$-complete (see Example \ref{ex:projective}).
Assume that $n+q-1$ is even.  By Theorem \ref{th:Hodge3} and Remark \ref{rem:finitely-gen} 
every element of $H^{n+q-1}(X;\Z)$ can be represented by a noncompact analytic cycle of complex dimension 
$p=(n-q+1)/2$ in $X$. However, if $p\ge q$ then a nonzero element of $H^{n+q-1}(X;\Z)$ 
can not be represented by a compact analytic $p$-cycle since every
such cycle intersects $A$ in view of Bezout's theorem (or by observing that a 
$q$-complete complex manifold does not contain any compact complex subvarieties of dimension 
$\ge q$). The inequality $p\ge q$ is equivalent to $n+1\ge 3q$. The lowest dimensional
non-Stein example arises when $q=2$ and $n=5$; in this case $X=\CP^5\setminus A$
where $A$ is a $3$-fold defined  by $2$ independent equations.
\end{example}

With $X$ as in Example  \ref{ex:projective2}, the top dimensional cohomology group 
$H^{n+q-1}(X;\Z)$ can be quite big and it can contain torsion as shown by the following proposition.

\begin{proposition}\label{prop:coh}
Let $A\subset\CP^n$ be a complete intersection of codimension $q$
and set  $X=\CP^n\setminus A$. Assume that $m:=\dim A=n-q\ge 1$ is odd. Then 
\[
	H^{n+q-1}(X;\Z)=\Z_d \oplus  \Z^{\beta_m}
\] 
where $d$ is the degree of $A$ and $\beta_m=m+1-\chi(A)$ is the $m$-th Betti number of $A$. 
\end{proposition}

The Euler characteristic $\chi(A)$ can be computed easily from the multidegree of $A$ by the Hirzebruch formula. 
For example, if $A$ is the intersection of two quintics in $\CP^5$ then its third Betti number is $4504$, so we have
$H^6(\CP^5\setminus A)\cong \Z_{5}\oplus \Z^{4504}$.

\begin{proof}
We omit the coefficient group $\Z$ in the calculations.
The Poincar\'e-Lefschetz duality gives $H^{n+q-1}(X)=H_{n-q+1}(\CP^n,A)=H_{m+1}(\CP^n,A)$.
The latter group can be computed from the exact homology sequence of the pair 
$i\colon A\hookrightarrow\CP^n$:
\[
	H_{m+1}(A) \stackrel{i_*}{\longrightarrow} H_{m+1}(\CP^n) \longrightarrow 
	H_{m+1}(\CP^n,A) \stackrel{\partial}{\longrightarrow} H_{m}(A) \longrightarrow H_{m}(\CP^n). 
\]	
Since $m$ is odd, we have $H_{m+1}(\CP^n)=\Z$ and $H_{m}(\CP^n)=0$. 
The Lefschetz hyperplane theorem yields (abstract) isomorphisms $H^k(A;\Z)\cong H^k(\CP^n;\Z)$
for $k\neq m$, $k\le 2m$. In  our case, with $k=m+1$ even, both groups equal $\Z$ and the map 
$\Z\stackrel{i_*}{\longrightarrow} \Z$ is multiplication by
$d$. Moreover, $H^m(A;\Z)$ is also free, $H^m(A;\Z)=\Z^{\beta_m}$.
This gives a short exact sequence $0\to \Z_d \to H_{m+1}(\CP^n,A) \to \Z^{\beta_m}\to 0$
from which the result follows.
\end{proof}

\begin{remark} 
It may be interesting to observe that the free part of $H^{n+q-1}(X;\Z)$ in 
Proposition \ref{prop:coh} is not representable by compact cycles even smoothly. 
To see this,  let $Z$ be an oriented closed smooth submanifold of 
$X$ of dimension $n-q+1$.
By compactness, there is a smooth tubular neighborhood $W$ of $Z$ in $\CP^n$ that is contained in $X$. 
The cohomology class dual to $Z$ is the image of the Thom class of the normal bundle under the composite 
\[	
	H^{n+q-1}(W,W\setminus Z;\Z)\xleftarrow{\rm{excision}}
       H^{n+q-1}(X,X\setminus Z;\Z)\xrightarrow{\rm{restriction}}H^{n+q-1}(X;\Z).	
\]
But a class in $H^{n+q-1}(X,X\setminus Z;\Z)$ can be excised back to a class in 
$H^{n+q-1}(\CP^n,\CP^n\setminus Z;\Z)$ and therefore the restriction
morphism $H^{n+q-1}(X,X\setminus Z;\Z)\to H^{n+q-1}(X;\Z)$ factors through 
$H^{n+q-1}(\CP^n;\Z)\to H^{n+q-1}(X;\Z)$. As shown in the
proof of Proposition \ref{prop:coh}, the image of the latter is precisely the torsion subgroup 
of $H^{n+q-1}(X;\Z)$ and misses the free subgroup completely.
\end{remark}

We conclude by considering  a couple examples of the form $X=Y\times \CP^1$ with $Y$ Stein,
so $X$ is $2$-complete. Let $Y$ denote the Stein surface
\begin{equation}\label{eq:Y}
	Y=\bigl\{[z_0:z_1:z_2]\in \C\P^2: z_0^2+z_1^2+z_2^2\ne 0 \bigr\}.
\end{equation}
For $k\in \N$ let $Y^k$ denote the $k$-th Cartesian power of $Y$, 
a Stein manifold of dimension $2k$. For every pair of integers 
$1\le j\le k$ the Chern class $c_j(Y^k)\in H^{2j}(Y^k;\Z)$ 
is the nonzero element of the group 
$H^{2j}(Y^k;\Z)\cong \Z_2$ (cf.\ Forster \cite[Proposition 3]{Forster1970}). 
In particular, $c_1(Y)\ne 0$ and the higher 
Chern classes of $Y$ vanish. (Compare with Example \ref{ex:Ex1}.)

\begin{example}\label{ex:Y}
Consider the $3$-manifold $X=Y\times \CP^1$, where $Y$ is given by (\ref{eq:Y}). 
Clearly $X$ is $2$-complete and has finite topology. From the K\"unneth theorem we
infer that 
\[ 
	H^4(X;\Z)\cong H^2(Y;\Z)\otimes H^2(\CP^1;\Z)\cong\Z_2.
\]
It is generated by the cohomological cross product $c_1(Y)\times \xi$ where
$\xi\in H^2(\CP^1;\Z)$ is the first Chern class of the tautological line bundle.
By Theorem \ref{th:Hodge3} and Remark \ref{rem:finitely-gen}, the nonzero element
of $H^4(X;\Z)$ can be represented by a noncompact analytic $1$-cycle. However,
it can not be represented by a compact analytic $1$-cycle.
Indeed, every compact analytic $1$-cycle in $X$ equals
$\sum_j\{y_j\}\times \CP^1$ for some $y_j\in Y$, and is homologous
to a multiple of $\{y_0\} \times \CP^1=:Z$ for any $y_0\in Y$. 
It is easily seen by standard arguments that the cohomology class in 
$H^4(X;\Z)$ dual to $Z$ (in the sense of Section \ref{sec:openMfd}) is trivial.  
\end{example}

\begin{example}\label{ex:Y2}
Let $Y$ be the surface (\ref{eq:Y}) and set $X=Y^2\times \CP^1$. We have
$n=\dim X=5$, $q=2$ (i.e., $X$ is 2-complete), and 
$n+q-1=6$. From the K\"unneth theorem we infer that 
\[
	H^6(X;\Z) \cong H^2(Y;\Z)\otimes H^2(Y;\Z)\otimes H^2(\CP^1;\Z)\cong\Z_2.
\]
It is generated by the cross product $c_1(Y)\times c_1(Y)\times\xi$ where $\xi\in H^2(\CP^1;\Z)$
is the first Chern class of the tautological line bundle.  By Theorem \ref{th:Hodge3}
and Remark \ref{rem:finitely-gen}, the nonzero element of $H^6(X;\Z)$ can be represented
by a noncompact analytic $2$-cycle in $X$. However, since $X$ is $2$-complete, it does not
admit any compact analytic $2$-cycles.

Similar results hold for the manifolds $X_{k,m}=Y^k\times \CP^m$ for higher values of $k,m\in \N$.
\end{example}

\medskip
\bf Acknowledgements. \rm
F.\ Forstneri\v c was supported in part by the Slovenian Research Agency (ARRS)
grants P1-0291 and J1-5432.  J.\ Smrekar was supported in part by the Slovenian
Research Agency grants P1-0292-0101 and J1-4144-0101.
A.\ Sukhov thanks the Mathematics Department of the University of Ljubljana and Institute of Mathematics, Physics and Mechanics
(IMFM), Ljubljana,  for several invitations and for excellent conditions during the work on this paper.

%%%%%%%%%%%%%%%%%%%%%%%%%%%%%%%%%%%%%%%%%%%%%%%%%%%%%%%%%%%%%%%%%%%%%%%%%
%																			    %
%  References                                                   										    %
%																                            %
%%%%%%%%%%%%%%%%%%%%%%%%%%%%%%%%%%%%%%%%%%%%%%%%%%%%%%%%%%%%%%%%%%%%%%%%%

\bibliographystyle{amsplain}

\end{document}